\documentclass[reqno, 12pt]{amsbook}
\usepackage{amsthm,amsmath,amsfonts,amssymb,amscd,mathrsfs,graphics}
\usepackage{txfonts,pxfonts}
\usepackage{hyperref,supertabular}
\usepackage{tikz}
\usepackage{stackrel}
\input xy
\xyoption{all}
\usepackage{booktabs}
\usepackage[disable]{todonotes}

\newfont{\bbb}{msbm10 scaled\magstep1} 
\newfont{\bbbsub}{msbm10}     
\newcommand{\M}{\overline{\mathcal{M}}} 
\newcommand{\D}{\mathcal{D}} 

\newcommand{\C}{\ensuremath{\mathbb{C}}}

\newcommand{\Z}{\ensuremath{\mathbb{Z}}}
\newcommand{\Q}{\ensuremath{\mathbb{Q}}}
\newcommand{\R}{\ensuremath{\mathbb{R}}}
\renewcommand{\P}{\ensuremath{\mathbb{P}}}
\renewcommand{\O}{\ensuremath{\mathcal{O}}}

\renewcommand{\d}{\ensuremath{\partial}}
\newcommand{\inv}{^{-1}}
\newcommand{\CCR}{\tb{CCR} }
\newcommand{\id}{\ensuremath{\textrm{id}}}

\newcommand{\teau}{\tilde{e}^\al}
\newcommand{\eau}{e^\al}
\newcommand{\tead}{\tilde{e}_\al}
\newcommand{\ead}{e_\al}

\newcommand{\ebu}{e^\beta}

\newcommand{\ebd}{e_\beta}

\newcommand{ \qa}{ {q}^\al}

\newcommand{ \pa}{ {p}_\al}
\newcommand{ \tqb}{ {\tilde{q}}^\beta}
\newcommand{ \qb}{ {q}^\beta}
\newcommand{ \tpb}{ {\tilde{p}}_\beta}
\newcommand{ \pb}{ {p}_\beta}

\newcommand{\del}{\partial}

\newcommand{\al}{\alpha}
\newcommand{\be}{\beta}

\newcommand{\ba}{\begin{array}}

\newcommand{\ea}{\end{array}}

\newcommand{\tb}{\textbf}

\DeclareMathOperator{\Aut}{Aut}
\DeclareMathOperator{\Sp}{Sp}
\DeclareMathOperator{\Res}{Res}
\DeclareMathOperator{\Id}{Id}
\DeclareMathOperator{\End}{End}

\newcommand{\pairing}[1]{\left\langle#1\right\rangle}

\newcommand{\mytilde}{{\raise.17ex\hbox{$\scriptstyle\mathtt{\sim}$}}}

\newtheorem{thm}{Theorem}[chapter]

\newtheorem{prop}[thm]{Proposition}
\newtheorem{crl}[thm]{Corollary}
\newtheorem{conj}[thm]{Conjecture}

\theoremstyle{definition}
\newtheorem{rem}[thm]{Remark}

\newtheorem{df}[thm]{Definition}
\newtheorem{ex}[thm]{Example}

\theoremstyle{remark}

\begin{document}

\pagestyle{empty}

\title{Geometric Quantization with Applications to Gromov-Witten Theory}
\author{Emily Clader, Nathan Priddis, and Mark Shoemaker}

\maketitle
\thispagestyle{empty}

\setcounter{chapter}{-1}

\noindent The notes that follow are expository, and the authors make no claim to originality of any of the material appearing in them.

\vspace{4in}
\noindent
Please send questions or comments to:\\
\\
\begin{tabular}{ll}
Emily Clader & eclader@umich.edu\\
Nathan Priddis & priddisn@umich.edu\\
Mark Shoemaker & markashoe@gmail.com
\end{tabular}

\tableofcontents

\pagestyle{headings}
\setcounter{page}{1}
\pagenumbering{arabic}

\chapter*{Preface}
The following notes were prepared for the ``IAS Program on Gromov-Witten Theory and Quantization'' held jointly by the Department of Mathematics and the Institute for Advanced Study at the Hong Kong University of Science and Technology in July 2013.  Their primary purpose is to introduce the reader to the machinery of geometric quantization with the ultimate goal of computations in Gromov-Witten theory.

First appearing in this subject in the work of Alexander Givental and his students, quantization provides a powerful tool for studying Gromov-Witten-type theories in higher genus.  For example, if the quantization of a symplectic transformation matches two total descendent potentials, then the original symplectic transformation matches the Lagrangian cones encoding their genus-zero theories; we discuss this statement in detail in Section \ref{semiclassical}.  Moreover, according to Givental's Conjecture (see Section \ref{Givental}), the converse is true in the semisimple case.  Thus, if one wishes to study a semisimple Gromov-Witten-type theory, it is sufficient to find a symplectic transformation identifying its genus-zero theory with that of a finite collection of points, which is well-understood.  The quantization of this transformation will carry all of the information about the higher-genus theory in question.

In addition, quantization is an extremely useful combinatorial device for organizing information.  Some of the basic properties of Gromov-Witten theory, for instance, can be succinctly expressed in terms of equations satisfied by quantized operators acting on the total descendent potential.  To give another example, even if one is concerned only with genus zero, the combinatorics of expanding the relation between two theories into a statement about their generating functions can be unmanageable, but when it is expressed via quantization this unwieldy problem obtains a clean expression.  The relationship between a twisted theory and its untwisted analogue, discussed in Section \ref{twistedtheory}, is a key instance of this phenomenon.

At present, there are very few references on the subject of quantization as it is used in this mathematical context.  While a number of texts exist (such as \cite{Blau}, \cite{Echeverria}, and \cite{Ritter}), these tend to focus on the quantization of finite-dimensional, topologically non-trivial symplectic manifolds.  Accordingly, they are largely devoted to explaining the structures of polarization and prequantization that one must impose on a symplectic manifold before quantizing, and less concerned with explicit computations.  These precursors to quantization are irrelevant in applications to Gromov-Witten theory, as the symplectic manifolds one must quantize are simply symplectic vector spaces.  However, other technical issues arise from the fact that the symplectic vector spaces in Gromov-Witten theory are typically infinite-dimensional.  We hope that these notes will fill a gap in the existing literature by focusing on computational formulas and addressing the complications specific to Givental's set-up.

The structure of the notes is as follows.  In Chapter 0, we give a brief overview of preliminary material on symplectic geometry and the method of Feynman diagram expansion.  We then turn in Chapter 1 to a discussion of quantization of finite-dimensional vector spaces.  We obtain formulas for the quantizations of functions on such vector spaces by three main methods: direct computation via the canonical commutation relations (which may also be expressed in terms of quadratic Hamiltonians), Fourier-type integrals, and Feynman diagram expansion.  All three methods yield equivalent results, but this diversity of derivations is valuable in pointing to various generalizations and applications.  We then include an interlude on basic Gromov-Witten theory.  Though this material is not strictly necessary until Chapter 4, 
it provides motivation and context for the material that follows.

The third chapter of the book is devoted to quantization of infinite-dimensional vector spaces, such as arise in applications to Gromov-Witten theory.  As in the finite-dimensional case, formulas can be obtained via quadratic Hamiltonians, Fourier integrals, or Feynman diagrams, and for the most part the computations mimic those of the first part.  The major difference in the infinite-dimensional setting though, is that issues of convergence arise, which we make an attempt to discuss whenever they come up.  Finally, in Chapter 4, we present several of the basic equations of Gromov-Witten theory in the language of quantization, and mention a few of the more significant appearances of quantization in the subject.

\subsection*{Acknowledgements}
The authors are indebted to Huai-Liang Chang, Wei-Ping Li, and Yongbin Ruan for organizing the workshop at which these notes were presented.  Special thanks are due to Yongbin Ruan for his constant guidance and support, as well as to Yefeng Shen for working through the material with us and assisting in the creation of these notes.  Y.P. Lee and Xiang Tang both gave extremely helpful talks at the RTG Workshop on Quantization organized by the authors in December 2011 at the University of Michigan.  Useful comments on earlier drafts were provided by Pedro Acosta and Weiqiang He.  A number of existing texts were used in the authors' learning process; we have attempted to include references to the literature throughout these notes, but we apologize in advance for any omissions. This work was supported in part by FRG grant DMS 1159265 RTG and NSF grant 1045119 RTG.

\chapter{Preliminaries}

Before we begin our study of quantization, we will give a quick overview of some of the prerequisite background material.  First, we review the basics of symplectic vector spaces and symplectic manifolds.  This material will be familiar to most of our mathematical audience, but we collect it here for reference and to establish notational conventions; for more details, see \cite{Cannas} or \cite{McDuff}.  Less likely to be familiar to mathematicians is the material on Feynman diagrams, so we cover this topic in more detail.  The chapter concludes with the statement of Feynman's theorem, which will be used later to express certain integrals as combinatorial summations.  The reader who is already experienced in the methods of Feynman diagram expansion is encouraged to skip directly to Chapter 1.

\section{Basics of symplectic geometry}

\subsection{Symplectic vector spaces}

A {\bf symplectic vector space} $(V,\omega)$ is a vector space $V$ together with a nondegenerate skew-symmetric bilnear form $\omega$.  We will often denote $\omega(v,w)$ by $\langle v, w\rangle$.

One consequence of the existence of a nondegenerate bilinear form is that $V$ is necessarily even-dimensional.  The standard example of a real symplectic vector space is $\R^{2n}$, with the symplectic form defined in a basis $\{e^{\alpha},e_{\beta}\}_{1 \leq \alpha, \beta \leq n}$ by
\begin{equation}\label{symp_basis}
\pairing{e^{\alpha}, e^{\beta}} = \pairing{e_{\alpha}, e_{\beta}} = 0,\quad \pairing{e^{\alpha}, e_{\beta}} = \delta^{\alpha}_{\beta}.
\end{equation}

In other words, $\langle \; , \; \rangle$ is represented by the (skew-symmetric) matrix
\[
J=\left(\begin{matrix}
0 & I\\
-I & 0
\end{matrix}\right).
\]
In fact, any finite-dimensional symplectic vector space admits a basis in which the symplectic form is expressed this way.  Such a basis is called a {\bf symplectic basis}, and the corresponding coordinates are known as {\bf Darboux coordinates}.

If $U \subset V$ is a linear subspace of a symplectic vector space $(V, \omega)$, then the {\bf symplectic orthogonal} $U^{\omega}$ of $U$ is
\[U^{\omega} = \{ v \in V \; | \; \omega(u,v) = 0 \text{ for all } u \in U\}.\]
One says $U$ is {\bf isotropic} if $U \subset U^{\omega}$ and {\bf Lagrangian} if $U = U^{\omega}$, which implies in particular that, if $V$ is finite-dimensional, $\text{dim}(U) = \frac{1}{2}\text{dim}(V)$. 

A {\bf symplectic transformation} between symplectic vector spaces $(V, \omega)$ and $(V', \omega')$ is a linear map $\sigma: V \rightarrow V'$ such that
\[\omega'( \sigma(v), \sigma(w) ) = \omega( v, w ).\]
In what follows, we will mainly be concerned with the case $V = V'$, and it will be useful to express the symplectic condition on a linear endomorphism $\sigma: V \rightarrow V$ in terms of matrix identities in Darboux coordinates.  Choose a symplectic basis for $V$ and express $\sigma$ in this basis via the matrix
\[
\sigma=\left(\begin{matrix}
A & B \\
C & D
\end{matrix}\right).
\]
Then $\sigma$ is symplectic if and only if $\sigma^TJ\sigma=J$, which in turn holds if and only if the following three identities are satisfied:
\begin{align}\label{smpl}
A^TB&=B^TA\\
C^TD&=D^TC\\
\label{smpl3}
A^TD-&B^TC=I.
\end{align}
Using these facts, one obtains a convenient expression for the inverse:
\[
\sigma^{-1}=\left(\begin{matrix}
D^T & -B^T \\
-C^T & A^T
\end{matrix}\right).
\]
An invertible matrix satisfying (\ref{smpl}) - (\ref{smpl3}) is known as a {\bf symplectic matrix}, and the group of symplectic matrices is denoted $\Sp(2n, \R)$.

\subsection{Symplectic manifolds}

Symplectic vector spaces are the simplest examples of the more general notion of a {\bf symplectic manifold}, which is a smooth manifold equipped with a closed nondegenerate two-form, called a {\bf symplectic form}.

In particular, such a $2$-form makes the tangent space $T_pM$ at any point $p \in M$ into a symplectic vector space.  Just as every symplectic vector space is isomorphic to $\R^{2n}$ with its standard symplectic structure, Darboux's Theorem states that every symplectic manifold is locally isomorphic to $\R^{2n} = \{(x_1, \ldots, x_n, y_1, \ldots, y_n)\}$ with the symplectic form $\omega = \sum_{i=1}^n dx_i \wedge dy_i$.

Perhaps the most important example of a symplectic manifold is the cotangent bundle.  Given any smooth manifold $N$ (not necessarily symplectic), there is a canonical symplectic struture on the total space of $T^*N$.  To define the symplectic form, let $\pi: T^*N \rightarrow N$ be the projection map.  Then one can define a one-form $\lambda$ on $T^*N$ by setting
\[
\lambda|_{\xi_x} = \pi^*(\xi_x)
\]
for any cotangent vector $\xi_x \in T_x^*N \subset T^*N$.  This is known as the {\bf tautological one-form}.  The {\bf canonical symplectic form} on $T^*N$ is $\omega = -d\lambda$.  The choice of sign makes the canonical symplectic structure agree with the standard one in the case where $N = \R^n$, which sits inside of $T^*N \cong \R^{2n}$ as $\text{Span}\{e^1, \ldots, e^n\}$ in the basis notation used previously.

In coordinates, the tautological one-form and canonical two-form appear as follows.  Let $q_1, \ldots, q_n$ be local coordinates on $N$.  Then there are local coordinates on $T^*N$ in which a point in the fiber over $(q_1, \ldots, q_n)$ can be expressed as a local system of coordinates $(q_1, \ldots, q_n, p^1, \ldots, p^n)$, and in these coordinates,
\[
\lambda = \sum_{\alpha} p^{\alpha} dq_{\alpha}
\]
and
\[
\omega = \sum_{\alpha} dq_{\alpha} \wedge dp^{\alpha}.
\]

The definitions of isotropic and Lagrangian subspaces generalize to symplectic manifolds, as well. A submanifold of a symplectic manifold is {\bf isotropic} if the restriction of the symplectic form to the submanifold is zero.  An isotropic submanifold is {\bf Lagrangian} if its dimension is as large as possible---namely, half the dimension of the ambient manifold.

For example, in the case of the symplectic manifold $T^*N$ with its canonical symplectic form, the fibers of the bundle are all Lagrangian submanifolds, as is the zero section.  Furthermore, if $\mu: N \rightarrow T^*N$ is a closed one-form, the graph of $\mu$ is a Lagrangian submanifold.  More precisely, for any one-form $\mu$, one can define
\[X_{\mu} = \{(x, \mu(x)) \; | \; x \in N\} \subset T^*N,\]
and $X_{\mu}$ is Lagrangian if and only if $\mu$ is closed.  In case $N$ is simply-connected, this is equivalent to the requirement that $\mu = df$ for some function $f$, called a {\bf generating function} of the Lagrangian submanifold $X_{\mu}$.

Finally, the notion of symplectic transformation generalizes in an obvious way.  Namely, given symplectic manifolds $(M, \omega_1)$ and $(N, \omega_2)$ a {\bf symplectomorphism} is a smooth map $f: M \rightarrow N$ such that $f^*\omega_2 = \omega_1$.

\section{Feynman diagrams}

The following material is drawn mainly from \cite{Etingof}; other references on the subject of Feynman diagrams include Chapter 9 of \cite{Hori} and Lecture 9 of \cite{Etingof}.

Consider an integral of the form
\begin{equation}\label{feynint}
\hbar^{-\frac{d}{2}}\int_Ve^{-S(x)/\hbar} dx,
\end{equation}
where $V$ is a $d$-dimensional vector space, $\hbar$ is a formal parameter, and
\begin{equation}
\label{S}
S(x) = \frac{1}{2}B(x,x) + \sum_{m \geq 0}\frac{g_m}{m!} B_m(x, \ldots, x)
\end{equation}
for a bilinear form $B$ and $m$-multilinear forms $B_m$ on $V$, where $g_m$ are constants.  The integral (\ref{feynint}) can be understood as a formal series in the parameters $\hbar$ and $g_m$.


\subsection{Wick's Theorem}

We begin by addressing the simpler question of computing integrals involving an exponential of a bilinear form without any of the other tensors.  

Let $V$ be a vector space of dimension $d$ over $\R$, and let $B$ be a positive-definite bilinear form on $V$.  Wick's theorem will relate integrals of the form
\[\int_V l_1(x) \cdots l_N(x) e^{-B(x,x)/2} dx,\]
in which $l_1, \ldots, l_N$ are linear forms on $V$, to pairings on the set $[2k] = \{1,2, \ldots, 2k\}$. By a {\bf pairing} on  $[2k]$, we mean a partition of the set into $k$ disjoint subsets, each having two elements.  Let $\Pi_{2k}$ denote the set of pairings on $[2k]$.  The size of this set is
\[\left|\Pi_{2k} \right|=\frac{(2k)!}{(2!)^k k!},\]
as the reader can check as an exercise.

An element $\sigma \in \Pi_k$ can be viewed as a special kind of permutation on the set $[2k]$, which sends each element to the other member of its pair.  Write $[2k]/\sigma$ for the set of orbits under this involution.

\begin{thm}[Wick's Theorem]
Let $l_1,\dots,l_N\in V^*$.  If $N$ is even, then
\[
\int_V l_1(x)\dots l_N(x)e^{\tfrac{-B(x,x)}{2}}\ dx=
\frac{(2\pi)^{d/2}}{\sqrt{\det B}}\sum_{\sigma\in \Pi_{N}} \prod_{i\in [N]/\sigma} B\inv(l_i, l_{\sigma(i)}).
\]
If $N$ is odd, the integral is zero.
\end{thm}

\begin{proof}
First, apply a change of variables such that $B$ is of the form $B(x,x)=x_1^2+\dots+x_d^2$. The reader should check that this change of variables changes both sides of the equation by a factor of $\det (P)$, where $P$ is the change-of-basis matrix, and thus the equality prior to the change of variables is equivalent to the result after the change.  Furthermore, since both sides are multilinear in elements of $V^*$ and symmetric in $x_1, \ldots, x_d$, we may assume $l_1=l_2=\dots=l_N=x_1$.  The theorem is then reduced to computing
\[
\int_V x_1^{N} e^{\tfrac{-(x_1^2+\dots+x_d^2)}{2}}\ dx.
\]
This integral indeed vanishes when $N$ is odd, since the integrand is an odd function. If $N$ is even, write $N = 2k$.  In case $k=0$, the theorem holds by the well-known fact that
\[
\int_{-\infty}^\infty e^{\tfrac{-x^2}{2}}\ dx= \sqrt{2\pi}.
\]
For $k > 0$, we can use this same fact to integrate out the last $d-1$ variables, by which we see that the claim in the theorem is equivalent to
\begin{equation}\label{wickprf}
\int_{-\infty}^\infty x^{2k} e^{-\tfrac{-x^2}{2}}\ dx= \sqrt{2\pi}\frac{(2k)!}{2^k k!}.
\end{equation}

To prove (\ref{wickprf}), we first make the substitution $y=\tfrac{x^2}{2}$.  Recall that the gamma function is defined by
\[
\Gamma(z) = \int_0^{\infty} y^{z-1} e^{-y} dy
\]
and satisfies $\Gamma(\frac{1}{2}) = \sqrt{\pi}$, as well as
\[
\Gamma(z+k) = (z+k-1)(z+k-2) \cdots z\cdot \Gamma(z)
\]
for integers $k$.  Thus, we have:
\begin{align*}
\int_ {-\infty}^\infty x^{2k} e^{\tfrac{-x^2}{2}}\ dx &= 2\int_{0}^\infty x^{2k} e^{\tfrac{-x^2}{2}}\ dx\\
&=2\int_{0}^\infty (2y)^{k-1/2} e^{-y}\ dy\\
&=2^{k+1/2}\Gamma(k+\tfrac 12)\\
&=2^{k+1/2}(k-\tfrac 12)(k-\tfrac 32)\dots(\tfrac 12)\Gamma(\tfrac 12)\\
&=\sqrt{2\pi} \frac{(2k)!}{2^k k!},
\end{align*}
which proves the claim.
\end{proof}

\subsection{Feynman's theorem}

Now let us return to the more general integral
\[Z = \hbar^{- \frac{d}{2}} \int_V e^{-S(x)/\hbar}dx,\]
which, for reasons we will mention at the end of the section, is sometimes called a {\bf partition function}.  Recall that $S$ has an expansion in terms of multilinear forms given by (\ref{S}).

Because a pairing can be represented by a graph all of whose vertices are $1$-valent, Wick's theorem can be seen as a method for expressing certain integrals as summations over graphs, in which each graph contributes an explicit combinatorial term.  The goal of this section is to give a similar graph-sum expression for the partition function $Z$.

Before we can state the theorem, we require a bit of notation.  If $\mathbf{n} = (n_0, n_1, \ldots)$ is a sequence of nonnegative integers, all but finitely many of which are zero, let $G(\mathbf{n})$ denote the set of isomorphism classes of graphs with $n_i$ vertices of valence $i$ for each $i \geq 0$.  Note that the notion of ``graph" here is very broad: they may be disconnected, and self-edges and multiple edges are allowed.  If $\Gamma$ is such a graph, let
\[
b(\Gamma) = |E(\Gamma)| - |V(\Gamma)|,
\]
where $E(\Gamma)$ and $V(\Gamma)$ denote the edge set and vertex set, respectively.  An {\bf automorphism} of $\Gamma$ is a permutation of the vertices and edges that preserves the graph structure, and the set of automorphisms is denoted $\text{Aut}(\Gamma)$.

We will associate a certain number $F_{\Gamma}$ to each graph, known as the {\bf Feynman amplitude}.
The Feynman amplitude is defined by the following procedure:
\begin{enumerate}
\item Put the $m$-tensor $-B_m$ at each $m$-valent vertex of $\Gamma$.
\item For each edge $e$ of $\Gamma$, take the contraction of tensors attached to the vertices of $e$ using the bilinear form $B^{-1}$.  This will produce a number $F_{\Gamma_i}$ for each connected component $\Gamma_i$ of $\Gamma$.
\item If $\Gamma = \bigsqcup_i \Gamma_i$ is the decomposition of $\Gamma$ into connected components, define $F_{\Gamma} = \prod_{i} F_{\Gamma_i}$.
\end{enumerate}
By convention, we set the Feynman amplitude of the empty graph to be $1$.

\begin{thm}[Feynman's Theorem]
One has
\[
Z = \frac{(2\pi)^{d/2}}{\sqrt{\det(B)}} \sum_{\mathbf{n} = (n_0, n_1, \ldots)} \left( \prod_{i=0}^{\infty} g_i^{n_i}\right) \sum_{\Gamma \in G(\mathbf{n})} \frac{\hbar^{b(\Gamma)}}{|\text{Aut}(\Gamma)|} F_{\Gamma},
\]
where the outer summation is over all sequences of nonnegative integers with almost all zero. 
\end{thm}

Before we prove the theorem, let us compute a few examples of Feynman amplitudes to make the procedure clear.

\begin{ex} Let $\Gamma$ be the following graph:
\tikzstyle vertex=[circle, draw, fill=black!50,
inner sep=0pt, minimum width=4pt]
\begin{figure}[h]
\centering
\begin{tikzpicture}[thick,scale=1]
\draw \foreach \x in {0,2}
{
(\x, 0) node[vertex] {} -- (\x+1,0) node[vertex] {}
};
\end{tikzpicture}
\label{dscnnctgrph1}
\end{figure}
\end{ex}

\noindent Given $B: V \otimes V \rightarrow \R$, we have a corresponding bilinear form $B^{-1}: V^{\vee} \otimes V^{\vee} \rightarrow \R$.  Moreover, $B_1 \in V^{\vee}$, and so we can write the Feynman amplitude of this graph as
\[F_{\Gamma} = (B^{-1}(-B_1, -B_1))^2.\]

\begin{ex} Consider now the graph
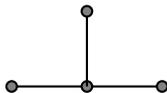
\tikzstyle vertex=[circle, draw, fill=black!50,
inner sep=0pt, minimum width=4pt]
\begin{figure}[h]
\centering
\begin{tikzpicture}[thick,scale=1]
\draw \foreach \x in {0,1}
{
(\x,0) node[vertex] {} -- (\x+1,0)
};
\draw {(1,0) -- (1,1) node[vertex] {} };
\node[vertex] at (2,0){};
\end{tikzpicture}
\label{tritree1}
\end{figure}
\end{ex}

\noindent Associated to $B$ is a map $\overline{B}: V \rightarrow V^{\vee}$, and in this notation, the Feynman amplitude of the graph can be expressed as
\[F_{\Gamma} = -B_3(\overline{B}^{-1}(-B_1), \overline{B}^{-1}(-B_1), \overline{B}^{-1}(-B_1)).\]

\begin{proof}[Proof of Feynman's Theorem]

After a bit of combinatorial fiddling, this theorem actually follows directly from Wick's theorem.  First, perform the change of variables $y = x/\sqrt{\hbar}$, under which
\[Z = \int_V e^{-B(y,y)/2} e^{\sum_{m \geq 0} g_m (\frac{-\hbar^{\frac{m}{2} -1} B_m(y,\ldots, y)}{m!})} dy.\]
Expanding the exponential as a series gives
\begin{align*}
Z &= \int_V e^{-B(y,y)/2} \prod_{i \geq 0} \sum_{n_i \geq 0} \frac{g_i^{n_i}}{(i!)^{n_i} n_i!}\big(-\hbar^{\frac{i}{2}-1} B_i(y, \ldots, y)\big)^{n_i} dy\\
&= \sum_{\mathbf{n} = (n_0, n_1, \ldots)} \bigg(\prod_{i \geq 0} \frac{g_i^{n_i}}{(i!)^{n_i} n_i!} \hbar^{n_i(\frac{i}{2}-1)}\bigg) \int_V e^{-B(y,y)/2} \prod_{i \geq 0} \big(-B_i(y, \ldots, y)\big)^{n_i} dy.
\end{align*}
Denote
\[Z_{\mathbf{n}} = \int_V e^{-B(y,y)/2} \prod_{i \geq 0} (-B_i(y, \ldots, y))^{n_i} dy.\]
Each of the factors $-B_i(y, \ldots, y)$ in this integral can be expressed as a sum of products of $i$ linear forms on $V$.  After unpacking the expresion in this way, $Z_{\mathbf{n}}$ becomes a sum of integrals of the form
\[
\int_V e^{-B(y,y)/2} \; \big(\text{one linear form}\big)^{n_1} \big(\text{product of two linear forms}\big)^{n_2}\cdots,
\]
each with an appropriate coefficient, so we will be able to apply Wick's theorem. 

Let $N = \sum_i i\cdot n_i$, which is the number of linear forms in the above expression for $Z_{\mathbf{n}}$. We want to express the integral $Z_{\mathbf{n}}$ using graphs. To this end, we draw $n_0$ vertices with no edges, $n_1$ vertices with $1$ half-edge emanating from them, $n_2$ vertices with $2$ half-edges, $n_3$ vertices with $3$ half-edges, et cetera; these are sometimes called ``flowers".  We place $-B_i$ at the vertex of each $i$-valent flower.  A pairing $\sigma \in \Pi_N$ can be understood as a way of joining pairs of the half-edges to form full edges, which yields a graph $\Gamma(\sigma)$, and applying Wick's theorem we get a number $F(\sigma)$ from each such pairing $\sigma$. One can check that all of the pairings $\sigma$ giving rise to a particular graph $\Gamma = \Gamma(\sigma)$ combine to contribute the Feynman amplitude $F_{\Gamma(\sigma)}$. In particular, by Wick's theorem, only even $N$ give a nonzero result.

At this point, we have
\[
Z = \frac{(2\pi)^{d/2}}{\sqrt{\det(B)}}\sum_{\mathbf{n}} \bigg(\prod_{i \geq 0} \frac{g_i^{n_i}}{(i!)^{n_i} n_i!}\bigg) \sum_{\sigma \in \Pi_N} \hbar^{b(\Gamma(\sigma)} F_{\Gamma(\sigma)},
\]
where we have used the straightforward observation that the exponent $\sum n_i(\frac{i}{2} - 1) = \frac{N}{2} - \sum n_i$ on $\hbar$ is equal to the number of edges minus the number of vertices of any of the graphs appearing in the summand corresponding to $\mathbf{n}$.  All that remains is to account for the fact that many pairings can yield the same graph, and thus we will obtain a factor when we re-express the above as a summation over graphs rather than over pairings.

To compute the factor, fix a graph $\Gamma$, and consider the set $P(\Gamma)$ of pairings on $[N]$ yielding the graph $\Gamma$.  Let $H$ be the set of half-edges of $\Gamma$, which are attached as above to a collection of flowers.  Let $G$ be the group of permutations of $H$ that perserve flowers; this is generated by permutations of the edges within a single flower, as well as swaps of two entire flowers with the same valence.  Using this, it is easy to see that
\[|G|= \prod_{i \geq 0} (i!)^{n_i} n_i!.\]
The group $G$ acts transitively on the set $P(\Gamma)$, and the stabilizer of this action is equal to $\text{Aut}(\Gamma)$.  Thus, the number of distinct pairings yielding the graph $\Gamma$ is
\[\frac{\prod_i (i!)^{n_i} n_i!}{|\text{Aut}(\Gamma)|},\]
and the theorem follows.
\end{proof}

We conclude this preliminary chapter with a bit of ``generatingfunctionology" that motivates the term ``partition function" for $Z$.

\begin{thm}
\label{connected}
Let $Z_0 = \frac{(2\pi)^{d/2}}{\det(B)}$.  Then one has
\[\log(Z/Z_0) = \sum_{\mathbf{n} = (n_0, n_1, \ldots )} \left(\prod_{i=0}^{\infty} g_i^{n_i}\right) \sum_{\Gamma \in G_c(\mathbf{n})} \frac{\hbar^{b(\Gamma)}}{|\text{Aut}(\Gamma)|} F_{\Gamma},\]
where the outer summation is over $\mathbf{n}$ as before, and $G_c(\mathbf{n})$ denotes the set of isomorphism classes of connected graphs in $G(\mathbf{n})$.
\end{thm}

The proof of this theorem is a combinatorial exercise and will be omitted.  Note that if we begin from this theorem, then the terms in Feynman's Theorem can be viewed as arising from ``partitioning" a disconnected graph into its connected components.   This explains the terminology for $Z$ in this particular situation.

\chapter{Quantization of finite-dimensional symplectic vector spaces} \label{finite}

From a physical perspective, geometric quantization arises from an attempt to make sense of the passage from a classical theory to the corresponding quantum theory.  The state space in classical mechanics is represented by a symplectic manifold, and the observables (quantities like position and momentum) are given by smooth real-valued functions on that manifold.  Quantum mechanics, on the other hand, has a Hilbert space as its state space, and the observables are given by self-adjoint linear operators.  Thus, quantization should associate a Hilbert space to each symplectic manifold and a self-adjoint linear operator to each smooth real-valued function, and this process should be functorial with respect to symplectic diffeomorphisms.

By considering certain axioms required of the quantization procedure, we show in Section \ref{setup} that the Hilbert space of states can be viewed as a certain space of functions.  Careful study of these axioms leads, in Section \ref{CCR}, to a representation of a quantized symplectic transformation as an explicit expression in terms of multiplication and differentiation of the coordinates.  However, as is explained in Section \ref{integrals}, it can also be expressed as a certain integral over the underlying vector space.  This representation has two advantages.  First, the method of Feynman diagrams allows one to re-write it as a combinatorial summation, as is explained in Section \ref{Feynman}.  Second, it can be generalized to the case in which the symplectic diffeomorphism is nonlinear.  Though the nonlinear case will not be addressed in these notes, we conclude this chapter with a few comments on nonlinear symplectomorphisms and other possible generalizations of the material developed here.

\section{The set-up}
\label{setup}

The material of this section is standard in the physics literature, and can be found, for example, in \cite{Blau} or \cite{Echeverria}.

\subsection{Quantization of the state space}

Let $V$ be a real symplectic vector space of dimension $2n$, whose elements are considered to be the classical states.  Roughly speaking, elements of the associated Hilbert space of quantum states will be square-integrable functions on $V$.  It is a basic physical principle, however, that quantum states should depend on only half as many variables as the corresponding classical states; there are $n$ position coordinates and $n$ momentum coordinates describing the classical state of a system, whereas a quantum state is determined by either position or momentum alone.  Thus, before quantizing $V$ it is necessary to choose a {\bf polarization}, a decomposition into half-dimensional subspaces.  Because these two subspaces should be thought of as position and momentum, which mathematically are the zero section and the fiber direction of the cotangent bundle to a manifold, they should be Lagrangian subspaces of $V$.

The easiest way to specify a polarization in the context of vector spaces is to choose a symplectic basis $e = \{e^{\alpha},e_{\beta}\}_{1 \leq \alpha, \beta \leq n}$.  Recall, such a basis satisfies
\[
\pairing{e^{\alpha}, e^{\beta}} = \pairing{e_{\alpha}, e_{\beta}} = 0,\quad \pairing{e^{\alpha}, e_{\beta}} = \delta^{\alpha}_{\beta},
\]
where $\langle \cdot, \cdot \rangle$ is the symplectic form on $V$.  The polarization may be specified by fixing the subspace $R = \text{Span}\{e_{\alpha}\}_{1 \leq \alpha \leq n}$ and viewing $V$ as the cotangent bundle $T^*R$ in such a way that $\langle \cdot, \cdot \rangle$ is identified with the canonical symplectic form.\footnote{Note that in this identification, we have chosen for the fiber coordinates to be the {\it first} $n$ coordinates.  It is important to keep track of whether upper indices or lower indices appear first in the ordering of the basis to avoid sign errors.}  An element of $V$ will be written in the basis $e$ as
\[
\sum_{\alpha} p_{\alpha} e^{\alpha} + \sum_{\beta} q^{\beta} e_{\beta}.
\]
For the remainder of these notes, we will suppress the summation symbol in expressions like this, adopting Einstein's convention that when Greek letters appear both up and down, they are automatically summed over all values of the index. For example, the above summation would be written simply as $p_\alpha e^\alpha + q^\beta e_\beta$. 

Let $V \cong \R^{2n}$ be a symplectic vector space with symplectic basis $e = \{e^{\alpha}, e_{\beta}\}_{1 \leq \alpha, \beta \leq n}$.  Then the {\bf quantization} of $(V,e)$ is the Hilbert space $\mathscr{H}_e$ of square-integrable functions on $R$ which take values in $\C[[\hbar, \hbar^{-1}]]$.  Here, $\hbar$ is considered as a formal parameter, although physically, it denotes Plank's constant.

It is worth noting that, while it is necessary to impose square-integrability in order to obtain a Hilbert space, in practice one often needs to consider formal functions on $R$ that are not square-integrable.  The space of such formal functions from $R$ to $\C[[\hbar, \hbar^{-1}]]$ is called the {\bf Fock space}.

\subsection{Quantization of the observables}

Observables in the classical setting are smooth functions $f \in C^{\infty}(V)$, and the result of a measurement is the value taken by $f$ on a point of $V$.  In the quantum framework, observables are operators $U$ on $\mathscr{H}_e$, and the result of a measurement is an eigenvalue of $U$.  In order to ensure that these eigenvalues are real, we require that the operators be self-adjoint.

There are a few other properties one would like the quantization $Q(f)$ of observables $f$ to satisfy.  We give one possible such list below, following Section 3 of \cite{Blau}.  Here, a set of observables is called {\bf complete} if any function that ``Poisson-commutes" (that is, has vanishing Poisson bracket) with every element of the set is a constant function.  Likewise, a set of operators is called {\bf complete} if any operator that commutes with every one of them is the identity.

The quantization procedure should satisfy:
\begin{enumerate}
\item {\bf Linearity}: $Q(\lambda f + g) = \lambda Q(f) + Q(g)$ for all $f,g \in C^{\infty}(V)$ and all $\lambda \in \R$.
\item {\bf Preservation of constants}: $Q(\mathbf{1}) = \id$.
\item {\bf Commutation}: $[Q(f), Q(g)]= \hbar Q(\{f,g\})$, where $\{ \; , \;\}$ denotes the Poisson bracket.\footnote{This convention differs by a factor of $i$ from what is taken in \cite{Blau}, but we choose it to match with what appears in the Gromov-Witten theory literature.}
\item {\bf Irreducibility}: If $\{f_1, \ldots, f_k\}$ is a complete set of observables, then $\{Q(f_1), \ldots, Q(f_k)\}$ is a complete set of operators.
\end{enumerate}
However, it is in general not possible to satisfy all four of these properties simultaneously.  In practice, this forces one to restrict to quantizing only a certain complete subset of the observables, or to relax the properties required.  We will address this in our particular case of interest shortly.

One complete set of observables on the state space $V$ is given by the coordinate functions $\{p_{\alpha}, q^{\beta}\}_{\alpha,\beta = 1,\ldots, n}$, and one can determine a quantization of these observables by unpacking conditions (1) - (4) above.  Indeed, when $f$ and $g$ are coordinate functions, condition (3) reduces to the {\bf canonical commutation relations} (\CCR), where we write $\hat{x}$ for $Q(x)$:
\[
[\hat{p}_{\alpha}, \hat{p}_{\beta}] = [\hat{q}^{\alpha}, \hat{q}^{\beta}] = 0, \; \; [\hat{p}_{\alpha}, \hat{q}^{\beta}] = \hbar \delta_{\alpha}^{\beta}.
\]
The algebra generated by elements $\hat{q}^{\alpha}$ and $\hat{p}_{\beta}$ subject to these commutation relations is known as the Heisenberg algebra.  Thus, the above can be understood as requiring that the quantization of the coordinate functions defines a representation of the Heisenberg algebra.  By Schur's Lemma, condition (4) is equivalent to the requirement that this representation be irreducible.  The following definition provides such a representation.

The {\bf quantization} of the coordinate functions\footnote{To be precise, these operators do not act on the entire quantum state space $\mathscr{H}_e$, because elements of $\mathscr{H}_e$ may not be differentiable.  However, this will not be an issue in our applications, because quantized operators will always act on power series.} is given by
\begin{align*}
\hat{q}^{\alpha}\Psi = q^{\alpha} \Psi,\\
\hat{p}_{\alpha} \Psi = \hbar\frac{\d \Psi}{\d q^{\alpha}}
\end{align*}
for $\Psi \in \mathscr{H}_e$.

In fact, the complete set of observables we will need to quantize for our intended applications consists of the quadratic functions in the Darboux coordinates.  To quantize these, order the variables of each quadratic monomial with $q$-coordinates on the left, and quantize each variable as above.  That is:
\begin{align*}
Q(q^{\alpha}q^{\beta}) &= q^{\alpha}q^{\beta}\\
Q(q^{\alpha}p_{\beta}) & = \hbar q^{\alpha} \frac{\d}{\d q^{\beta}}\\
Q(p_{\alpha}p_{\beta}) & = \hbar^2 \frac{\d^2}{\d q^{\alpha} \d q^{\beta}}.
\end{align*}

There is one important problem with this definition: the commutation condition (3) holds only up to an additive constant when applied to quadratic functions.  More precisely,
\[[Q(f) Q(g)] = \hbar Q(\{f,g\}) + \hbar^2 \mathcal{C}(f,g),\]
where $\mathcal{C}$ is the cocycle given by
\begin{equation}
\label{cocycle}
\mathcal{C}(p_{\alpha}p_{\beta}, q^{\alpha} q^{\beta}) = \begin{cases}1 & \alpha \neq \beta\\ 2 & \alpha = \beta\end{cases}
\end{equation}
and $\mathcal{C}(f,g) = 0$ for any other pair of quadratic monomials $f$ and $g$.  This ambiguity is sometimes expressed by saying that the quantization procedure gives only a projective representation of the Lie algebra of quadratic functions in the Darboux coordinates.

\subsection{Quantization of symplectomorphisms}

All that remains is to address the issue of functoriality.  That is, a symplectic diffeomorphism $T: V \rightarrow V$ should give rise to an operator $U_T: \mathscr{H}_{e} \rightarrow \mathscr{H}_{\tilde{e}}$.  In fact, we will do something slightly different: we will associate to $T$ an operator $U_T: \mathscr{H}_{e} \rightarrow \mathscr{H}_{e}$.  In certain cases, such as when $T$ is upper-triangular, there is a natural identification between $\mathscr{H}_e$ and $\mathscr{H}_{\tilde{e}}$, so our procedure does the required job.  More generally, the need for such an identification introduces some ambiguity into the functoriality of quantization.

Furthermore, we will consider only the case in which $T$ is linear and is the exponential of an infinitesimal symplectic transformation, which is simply a linear transformation whose exponential is symplectic. For applications to Gromov-Witten theory later, such transformation are the only ones we will need to quantize.

The computation of $U_T$ is the content of the next three sections.

\section{Quantization of symplectomorphisms via the \CCR}
\label{CCR}

The following section closely follows the presentation given in \cite{XT}.

Let $T: V \rightarrow V$ be a linear symplectic isomorphism taking the basis $e$ to a new basis $\tilde{e} = \{\tilde{e}^{\alpha}, \tilde{e}_{\beta}\}$ via the transformation
\[
\teau = A^\al_\beta \ebu + C^{\beta\al} \ebd,
\]
\[
\tead = B_{\beta\al} \ebu + D^\beta_\al \ebd.
\]
Let $\tilde{p}_{\alpha}$ and $\tilde{q}^{\beta}$ be the corresponding coordinate functions for the basis vectors $\tilde{e}^{\alpha}$ and $\tilde{e}_{\beta}$. Then $\widehat{ \tilde{\pb}}$ and $\widehat{ \tilde{\qa}}$ are defined in the same way as above.  The relation between the two sets of coordinate functions is:
\begin{align}
\label{coordinates}
 \pa &=  A_\al^\beta  \tpb + B_{\al \beta}  \tqb  \\
 \qa &=   C^{\al \beta}  \tpb + D^\al_\beta  \tqb , \nonumber
\end{align}
and the same equations give relations between their respective quantizations.  We will occasionally make use of the matrix notation
\begin{align*}
p &= A \tilde{p} + B \tilde{q}\\
q &= C\tilde{p} + D\tilde{q}
\end{align*}
to abbreviate the above.

To define the operator $\mathscr{H}_e \rightarrow \mathscr{H}_e$ associated to the transformation $T$, observe that by inverting the relationship \eqref{coordinates} and quantizing, one can view both $\hat{p}_{\alpha}, \hat{q}^{\beta}$ and $\widehat{\tilde{p}}_{\alpha}, \widehat{\hat{q}^{\beta}}$ as representations of $\mathscr{H}_e$.  The operator $U_T$ will be defined by the requirement that
\begin{align}
\label{satisfy}
{\hat{\tilde{q}}^\al} U_T &= U_T  \hat{  \qa} \\  
\label{satisfy2}
{\hat{\tilde{p}}_\al} U_T &= U_T  \hat{ \pa}.
\end{align}
As the computation will show, these equations uniquely specify $U_T$ up to a multiplicative constant.

To obtain an explicit formula for $U_T$, we restrict as mentioned above to the case in which
\[
T = \left(\begin{matrix}
A & B\\C & D
\end{matrix}\right) 
= \exp \left(\begin{matrix}
a & b\\ c & d
\end{matrix}\right).
\]
In other words, if
\[
T_t = \left(\begin{matrix}
A(t) & B(t)\\C(t) & D(t)
\end{matrix}\right)
=\exp \left(\begin{matrix}
t a & t b\\ t c & td
\end{matrix}\right),
\]
then $T$ is the derivative at $t=0$ of the family of transformations $T_t$.  The matrix $ \left(\begin{smallmatrix}a & b\\ c & d \end{smallmatrix}\right)$ satisfies the infinitesimal symplectic relations
\begin{equation}\label{infrel}
b^T=b, \qquad c^T=c, \qquad a^T=-d, 
\end{equation}
which follow from the relations defining $T$ by differentiation.

This family of transformations yields a family of bases $\{\eau_t, \ead^t\}$ defined by 
\begin{align*} 
\eau_t &= A(t)^\al_\beta \ebu + C(t)^{\beta\al} \ebd \\ 
\ead^t &= B(t)_{\beta\al} \ebu + D(t)^\beta_\al \ebd,
\end{align*}
with corresponding coordinate functions $\{\pa^t, \qa_t\}$.  Note that we obtain the original transformation $T$, as well as the original basis $\tilde{e}$ and coordinate functions $\tilde{p}_{\alpha}, \tilde{q}^{\beta}$, by setting $t=1$.

Using the fact that
\[
T^{-1}_t = \exp\left(\begin{matrix} 
-at & -bt\\ -ct & -dt
\end{matrix}\right) 
= \left(\begin{matrix}
I - ta & - tb \\ -tc & I - td
\end{matrix}\right) + O(t^2),
\]
we obtain the relations
\begin{align*}
e^{\alpha} &= e^{\alpha}_t - ta^{\alpha}_{\beta} e^{\beta}_t - tc^{\beta\alpha} e^t_{\beta} + O(t^2),\\
e_{\alpha} &= e_{\alpha}^t - tb_{\beta\alpha}e^{\beta}_t - td^{\beta}_{\alpha}e^t_{\beta} + O(t^2).
\end{align*}
This implies
\begin{align*}
 \pa^t &=  \pa  - t a_\al^\beta  \pb - t b_{\al \beta}  \qb + O(t^2)\\
 \qa_t &=  \qa - t c^{\al \beta}  \pb - t d^\al_\beta  \qb + O(t^2)
\end{align*}
and consequently
\begin{align}
 \hat{\pa^t} =  \hat{\pa}  - t a_\al^\beta  \hat{\pb} - t b_{ \al \beta}  \hat{\qb} + O(t^2) \label{xbd} \\
 \hat{\qa_t} =  \hat{\qa}  - t c^{ \al \beta}  \hat{\pb} - t d^\al_\beta  \hat{\qb} + O(t^2). \nonumber
\end{align}

Let $U_t = U_{T_t}$, so that
\begin{align}
\hat{\pa^t} U_t &= U_t  \hat{\pa}\label{
To depart while seated or standing is all one.
All I shall leave behind me
Is a heap of bones.
In empty space I twist and soar
And come down with the roar of thunder
To the sea.} \\
\hat{\qa_t} U_t &= U_t  \hat{\qa}. \nonumber
\end{align}
Denote by $u$ the {\bf infinitesimal variation of $U_T$}, 
\[\left. u  = \frac{d}{dt}\right|_{t=0} U_t.
\]
By plugging \eqref{xbd} 
into the above equations and taking the derivative at $t= 0$, we derive commutation relations satisfied by $u$:
\begin{align*} 
[\hat{ \pa},u] &= a^\beta_\al  \hat{\pb} + b_{  \al \beta}  \hat{\qb} \\ 
 [\hat{\qa},u] &= c^{ \al \beta}  \hat{\pb} + d^\al_\beta  \hat{\qb}.
\end{align*}
These equations allow us to determine the infinitesimal variation of $U$ up to a constant:
\[ 
u =  -\frac{1}{2 \hbar}  c^{\alpha\beta}  \hat{\pa} \hat{\pb} +  \frac{1}{\hbar}  a^{\beta}_{\alpha}  \hat{\qa} \hat{\pb} +  \frac{1}{2\hbar} b_{\alpha\beta} \hat{\qa}  \hat{\qb} + C.
\]
After identifying $\hat{\pa}$ and $\hat{\qa}$ with the operators  $\hbar\frac{\del }{\del q^\al}$  and $\qa$, respectively, we obtain
\begin{align}\label{don't stop believing}
u =  -\frac{ \hbar}{2}  c^{\alpha\beta}  \frac{\del }{\del q^\al} \frac{\del}{\del q^\beta} +  a_{\alpha}^{\beta}  \qa  \frac{\del}{\del q^\beta} + \frac{1}{2\hbar} b_{\alpha\beta} \qa  \qb + C.
\end{align}

Expanding \eqref{
To depart while seated or standing is all one.
All I shall leave behind me
Is a heap of bones.
In empty space I twist and soar
And come down with the roar of thunder
To the sea.} with respect to $t$ yields formulas for $\left[\hat{p_\alpha}, \left. \left( (\tfrac{d}{dt})^k U_t\right) \right|_{t=0}\right]$ and $\left[\hat{q^\alpha}, \left. \left( (\tfrac{d}{dt})^k U_t\right) \right|_{t=0}\right]$.  Using these, one can check that if $T_t$ takes the form
 $T_t = exp\left(t \left(\begin{smallmatrix}a & b\\ c & d \end{smallmatrix}\right)\right)$, then $U_t$ will take the form $U_t = \exp(tu)$, so in particular, $U_T = \exp(u)$.  Thus, equation (\ref{don't stop believing}) gives a general formula for $U_T$.

The formula simplifies significantly when $T$ takes certain special forms.  Let us look explicitly at some particularly simple cases.  
\begin{ex}
Consider first the case where $b=c=0$.  In this case, we obtain
\[
(U_T\psi)(q) = \exp\left(  a^{\beta}_{\alpha} q^\al \frac{\del}{\del q^\beta} \right)\psi(q) = \psi(A^Tq).
\]

Let us verify this formula in a very easy case.  Suppose that $a$ has only one nonzero entry, and that this entry is off the diagonal, so that $a^{\beta}_{\alpha} = \delta^i_{\alpha}\delta^{\beta}_j$ for some fixed $i \neq j$.  Assume furthermore that $\psi$ is a monomial:
\[
\psi(q) = \prod_{\alpha} (q^{\alpha})^{c_{\alpha}}.
\]
Then
\begin{align*}
(U_T\psi)(q) &= \exp\left(  a^j_i q^i \frac{\del}{\del q^j} \right)\prod \left(q^\al\right)^{c_\al} \\
 	&= \sum_{k = 0}^{c_j} \frac{1}{k!}\left(  a^j_i q^i \frac{\del}{\del q^j} \right)^k\prod \left(q^\al\right)^{c_\al} \\
 	&= \sum_{k = 0}^{c_j} \frac{c_j!}{(k!)(c_j - k)!}\left(a^j_i \right)^k (q^i)^{k} (q^j)^{c_j - k} \prod_{\al \neq j}\left(q^\al\right)^{c_\al}\\
 	&= \prod_{\al \neq j}\left(q^\al\right)^{c_\al}\left(q^j + a^j_i q^i\right)^{c_j} \\ 
 	&= \psi(A^Tq),
\end{align*}
where we use that $A = \exp(a) = I + a$ in this case.
\end{ex}

\begin{ex}
In the case where $a=c= 0$, or in other words, $T$ is upper-triangular, the formula above directly implies
\[
(U_T\psi)(q) = \exp\left( \frac{1}{2\hbar}b_{\al \be} q^\al q^\beta \right)\psi(q).
\]
However, in this case, it is easy to check that $B = b$, and we obtain
\[
(U_T\psi)(q) = \exp\left( \frac{1}{2\hbar}B_{\al \be} q^\al q^\beta \right)\psi(q).
\]
\end{ex}

\begin{ex}
Finally, consider the case where $a=b=0$, so that $T$ is lower-triangular.  Then
\[
(U_T\psi)(q) = \exp\left(-\frac{ \hbar}{2} C^{\al \be} \frac{\del}{\del q^\al} \frac{\del}{\del q^\beta} \right)\psi(q).
\]
\begin{rem}\label{lowert}
It is worth noting that this expression can be evaluated using Feynman diagram techniques, in which each diagram corresponds to a term in the Taylor series expansion of the exponential; we will discuss this further in Section~\ref{ltc}.
\end{rem}
\end{ex}

In fact, the above three examples let us compute $U_T$ for any symplectic matrix $T= \left(\begin{array}{cc} A& B \\ C&D \end{array} \right) = \exp\left(\begin{array}{cc} a& b\\ c & d \end{array}\right)$ for which the lower-right submatrix $D$ is invertible.\footnote{Throughout this text, we will assume for convenience that $D$ is invertible.  However, if this is not the case, one can still obtain similar formulas by decomposing the matrix differently.}  To do so, decompose $T$ as
\[T = \left(\begin{array}{cc} A & B\\ C & D\end{array}\right) = \left(\begin{array}{cc}I & BD^{-1}\\ 0 & I \end{array}\right)\cdot \left(\begin{array}{cc}I & 0 \\ DC^T & I \end{array}\right)\cdot\left(\begin{array}{cc}D^{-T} & 0 \\ 0 & D\end{array}\right).\]
Each of the matrices on the right falls into one of the cases we calculated above.  Furthermore, the quantization procedure satisfies $U_{T_1 \circ T_2} = U_{T_1} \circ U_{T_2}$ (up to a constant), since both sides satisfy (\ref{satisfy}) and (\ref{satisfy2}).  Thus, the formula for the quantization $U_T$ of any such $T$ is as follows:
\[(U_T\psi)(q) = \exp\left(\frac{1}{2\hbar}(BD^{-1})_{\alpha\beta}q^{\alpha}q^{\beta}\right) \exp\left(-\frac{\hbar}{2}(DC^T)^{\alpha\beta}\frac{\d}{\d q^{\alpha}} \frac{\d}{\d q^{\beta}}\right)\psi((D^{-T})^{\alpha}_{\beta}q),\]
or, in matrix notation,
\[(U_T\psi)(q) = \exp\left(\frac{1}{2\hbar}(BD^{-1}q)\cdot q\right)\exp\left(-\frac{\hbar}{2}\left(DC^T\frac{\d}{\d q}\right)\cdot \frac{\d}{\d q}\right)\psi(D^{-1}q).\]

One should be somewhat careful with this expression, since the two exponentials have, respectively, infinitely many negative powers of $\hbar$ and infinitely many positive powers of $\hbar$, so {\it a priori} their composition may have some powers of $\hbar$ whose coefficients are divergent series.  Avoiding this issue requires one to apply each quantized operator to $\psi(D^{-1}q)$ in turn, verifying at each stage that the coefficient of every power of $\hbar$ converges.  A similar issue will arise when dealing with powers of the variable $z$ in the infinite-dimensional setting; see Section \ref{convergence} of Chapter 3.

\section{Quantization via quadratic Hamiltonians}
\label{Ham}

Before moving on to other expressions for the quantization $U_T$, let us briefly observe that the formulas obtained in the previous section can be described in a much simpler fashion by referring to the terminology of Hamiltonian mechanics.  We have preferred the longer derivation via the \CCR because it more clearly captures the ``obvious" functoriality one would desire from the quantization procedure, but the Hamiltonian perspective is the one that is typically taken in discussions of quantization in Gromov-Witten theory (see, for example, \cite{Coates} or \cite{GiventalQuadraticHamiltonians}).

Let $T = \exp(F)$ be a symplectomorphism as above, where
\[
F = \left(\begin{matrix} 
a & b\\ c & d 
\end{matrix}\right)
\]
is an infinitesimal symplectic transformation.  Because the tangent space to a symplectic vector space at any point is canonically identified with the vector space itself, we can view $F: V \rightarrow V$ as a vector field on $V$.  If $\omega$ is the $2$-form giving the symplectic structure, the contraction $\iota_F \omega$ is a $1$-form on $V$.  Since $V$ is topologically contractible, we can write $\iota_F \omega = dh_F$ for some function $h_F: V \rightarrow \R$.  This function is referred to as the {\bf Hamiltonian} of $F$.  Concretely, it is described by the formula
\[
h_F(v) = \frac{1}{2} \langle Fv, v \rangle
\]
for $v \in V$, where $\langle \;, \; \rangle$ is the symplectic pairing.

Being a classical observable, the quantization of the function $h_F: V \rightarrow \R$ has already been defined.  Define the quantization of $F$ by
\[
\hat{F} = \frac{1}{\hbar} \hat{h_F}.
\]
The quantization of the symplectomorphism $T$ is then defined as
\[
U_T = \exp(\hat{F}).
\]

It is an easy exercise to check that $\hat{F}$ agrees with the general formula given by \eqref{don't stop believing}, so the two definitions of $U_T$ coincide.

One advantage of the Hamiltonian perspective is that it provides a straightforward way to understand the noncommutativity of the quantization procedure for infinitesimal symplectic transformations.  Recall, the quantization of quadratic observables obeys the commutation relation
\[[Q(f), Q(g)] = \hbar Q(\{f,g\}) + \hbar^2 \mathcal{C}(f,g)\]
for the cocycle $\mathcal{C}$ defined in \eqref{cocycle}.  It is easy to check (for example, by working in Darboux coordinates) that the Hamiltonian $h_A$ associated to an infinitesimal symplectic transformation satisfies
\[\{h_A, h_B\} = h_{[A,B]}.\]
Thus, the commutation relation for infinitesimal symplectic transformations is
\[[\hat{A}, \hat{B}] = \widehat{[A,B]} +  \mathcal{C}(h_A, h_B).\]
For an explicit computation of this cocycle in an (infinite-dimensional) case of particular interest, see Example 1.3.4.1 of \cite{Coates}.

\section{Integral formulas}
\label{integrals}

The contents of this section are based on lectures given by Xiang Tang at the RTG Workshop on Quantization at the University of Michigan in December 2011.  In a more general setting, the material is discussed in \cite{BatesWeinstein}.

Our goal is to obtain an alternative expression for $U_T$ of the form
\begin{equation}
\label{intformula}
(U_T\psi)(q) = \lambda \int_{\R^n} \int_{\R^n} e^{\frac{1}{\hbar}(\phi(q, p') - q' \cdot p')} \psi(q')dq'dp'
\end{equation}
for a function $\phi: \R^{2n} \rightarrow \R$ and a constant $\lambda \in \R[[\hbar, \hbar^{-1}]]$ to be determined.  Such operators, since they generalize the Fourier transform of $\psi$, are known as {\bf Fourier integral operators}.

The advantage of this alternate expression for $U_T$ is twofold.  First, they allow quantized operators to be expressed as sums over Feynman diagrams, and this combinatorial expansion will be useful later, especially in Section \ref{semiclassical}.  Second, the notion of a Fourier integral operator generalizes to the case when the symplectic diffeomorphism is not necessarily linear, as well as to the case of a Lagrangian submanifold of the cotangent bundle that is not the graph of any symplectomorphism; we will comment briefly on these more general settings in Section \ref{generalizations} below.

To define $\phi$, first let $\Gamma_T$ be the graph of $T$:
\[
\Gamma_T = \left\{(p, q, \tilde{p}, \tilde{q})\; \left| \; \begin{array}{c} p_{\alpha} = A_{\alpha}^{\beta} \tilde{p}_{\beta} + B_{\alpha\beta}\tilde{q}^{\beta}\\ q^{\alpha} = C^{\alpha\beta}\tilde{p}_{\beta} + D^{\alpha}_{\beta} \tilde{q}^{\beta} \end{array}\right.\right\} \subset \overline{\R^{2n}} \times \R^{2n}.
\]
Here, $\overline{\R^{2n}}$ denotes the symplectic vector space obtained by equipping $\R^{2n}$ with the opposite of the standard symplectic form, so that the symplectic form on the product is given by
\[
\langle(p, q, \tilde{p}, \tilde{q}), (P, Q, \tilde{P}, \tilde{Q})\rangle = \sum_i\big(-p_iQ^i + P_iq^i + \tilde{p}_i \tilde{Q}^i - \tilde{P}_i\tilde{q}^i\big).
\]
Under this choice of symplectic form, $\Gamma_T$ is a Lagrangian submanifold of the product.

There is an isomorphism of symplectic vector spaces
\begin{align*}
\overline{\R^{2n}} \times \R^{2n} &\xrightarrow{\sim} T^*(\R^{2n})\\
(p,q,\tilde{p}, \tilde{q}) &\mapsto (q, \tilde{p}, p, \tilde{q}),
\end{align*}
where $T^*(\R^{2n})$ is equipped with the canonical symplectic form
\[
\langle(q, p, \pi, \xi), (Q, P, \Pi, \Xi)\rangle = \sum_i \big(q^i \Pi_i + p_i\Xi^i - Q^i\pi_i - P_i\xi^i\big).
\]
Thus, one can view $\Gamma_T$ as a Lagrangian submanifold of the cotangent bundle.  Define $\phi = \phi(q, p')$ as the generating function for this submanifold.  Explicitly, this says that
\[
 \left\{(q, \tilde{p}, p, \tilde{q}))\; \left| \; \begin{array}{c} p_{\alpha} = A_{\alpha}^{\beta} \tilde{p}_{\beta} + B_{\alpha\beta}\tilde{q}^{\beta}\\ q^{\alpha} = C^{\alpha\beta}\tilde{p}_{\beta} + D^{\alpha}_{\beta} \tilde{q}^{\beta} \end{array}\right.\right\} = \left\{ \left(q, p',  \frac{\d \phi}{\d q}, \frac{\d \phi}{\d p'}\right)\right\}.
\]

Let us restrict, similarly to Section \ref{CCR}, to the case in which $D$ is invertible.  The relations defining $\Gamma_T$ can be rearranged to give
\begin{align*}
\tilde{q} &=D^{-1}q - D^{-1}C\tilde{p}\\
p &= BD^{-1}q + D^{-T}\tilde{p}.
\end{align*}
Therefore, $\phi(q, p')$ is defined by the system of partial differential equations
\begin{align*}
\frac{\d \phi}{\d q} &= BD^{-1}q + D^{-T} p'\\
\frac{\d \phi}{\d p'} &= D^{-1} q - D^{-1}Cp'.
\end{align*}
These are easily solved; up to an additive constant, one obtains
\[
\phi(q,p') =\frac{1}{2}(BD^{-1}q)\cdot q + (D^{-1}q) \cdot p' - \frac{1}{2}(D^{-1}Cp')\cdot p'.
\]

The constant $\lambda$ in the definition of $U_T$ is simply a normalization factor, and is given by
\[
\lambda = \frac{1}{\hbar^n},
\]
as this will be necessary to make the integral formulas match those computed in the previous section.  We thus obtain the following definition for $U_T$:
\begin{equation}
\label{integralformula}
(U_T\psi)(q) = \frac{1}{\hbar^n} \int_{\R^n}\int_{\R^n} e^{\frac{1}{\hbar}( \frac{1}{2}(BD^{-1}q)\cdot q + (D^{-1}q) \cdot p' - \frac{1}{2}(D^{-1}Cp')\cdot p' -q' \cdot p')} \psi(q')dq' dp'.
\end{equation}

We should verify that this formula agrees with the one obtained in Section \ref{CCR}.  This boils down to properties of the Fourier transform, which we will define as
\[\hat{\psi}(y) = \int_{\R^n} e^{-ix \cdot y} \psi(x) dx.\]
Under this definition,
\begin{align*}
(U_T\psi)(q) &= \frac{ e^{\frac{1}{2\hbar}(BD^{-1}q)\cdot q}}{\hbar^n} \int_{\R^n}\int_{\R^n} e^{\frac{1}{\hbar}(D^{-1}q)\cdot p'} e^{-\frac{1}{2\hbar}(D^{-1}Cp')\cdot p'} e^{-\frac{1}{\hbar}q' \cdot p'} \psi(q')dq'dp'\\
&=i^n e^{\frac{1}{2\hbar}(BD^{-1}q)\cdot q}\int_{i\R^n} e^{i(D^{-1}q)\cdot p''} e^{\frac{\hbar}{2}(D^{-1}C p'') \cdot p''} \hat{\psi}(p'') dp''\\
&= e^{\frac{1}{2\hbar}(BD^{-1}q)\cdot q}\int_{\R^n} e^{i Q \cdot p'''} e^{-\frac{\hbar}{2} (D^{-1}Cp''') \cdot p'''} \widehat{\psi \circ i}(p''') dp'''\\
&= e^{\frac{1}{2\hbar}(BD^{-1}q)\cdot q}\int_{\R^n} e^{i Q \cdot p'''} \bigg(e^{-\frac{\hbar}{2} (D^{-1}C\frac{\d}{\d q'}) \cdot \frac{\d}{\d q'}} (\psi) \circ i\bigg)^{\wedge}(p''') dp'''\\
&=e^{\frac{1}{2\hbar}(BD^{-1}q)\cdot q} e^{-\frac{\hbar}{2}\left(DC^T \frac{\d}{\d q}\right)\cdot \frac{\d}{\d q}} \psi(D^{-1}q),
\end{align*}
where we use the changes of variables $\frac{1}{\hbar}p' = ip''$, $ip'' = p'''$, and $D^{-1}q = iQ$.  The integral formula therefore matches the one defined via the \CCR.

\section{Expressing integrals via Feynman diagrams}
\label{Feynman}

The formula given for $U_T\psi$ in \eqref{integralformula} bears a striking resemblace to the type of integral computed by Feynman's theorem, and in this section, we will make the connection precise.

\subsection{Genus-modified Feynman's Theorem}

In order to apply Feynman's theorem, the entire integrand must be an exponential, so we will assume that
\[
\psi(q) = e^{\frac{1}{\hbar}f(q)}.
\]
Furthermore, let us assume that $f(q)$ is $\hbar^{-1}$ times a power series in $\hbar$, so $\psi$ is of the form
\[
\psi(q) = e^{\sum_{g \geq 0} \hbar^{g-1} \mathcal{F}_g(q)}.
\]
Then
\[
(U_T\psi)(q) = e^{\frac{1}{2\hbar}(BD^{-1}q) \cdot q}\; \hbar^{-n} \int_{\R^{2n}} e^{-\frac{1}{\hbar}S(p',q')}dp'dq'
\]
with
\[
S(p',q') = \bigg(- (D^{-1}q) \cdot p' + \frac{1}{2}(D^{-1}Cp') \cdot p' + q'\cdot p'\bigg) - \sum_{g \geq 0} \hbar^g \mathcal{F}_g(q').
\]
Note that if we let $y = (p',q')$, then the bilinear leading term of $S(p',q')$ (in parentheses above) is equal to 
\[
\frac{1}{2}(D^{-1}Cp') \cdot p' + \frac{1}{2}(q' - D^{-1}q) \cdot p' + \frac{1}{2}p' \cdot (q' - D^{-1}q) = \frac{\mathcal{B}(y - D^{-1}q, y - D^{-1}q)}{2},
\]
where $\mathcal{B}(y_1, y_2)$ is the bilinear form given by the block matrix
\[
\left(\begin{matrix} 
D^{-1}C & I\\ I & 0
\end{matrix}\right)
\]
and $q = (0, q) \in \R^{2n}$.

Changing variables, we have
\[
(U_T\psi)(q) = e^{\frac{1}{2\hbar}(BD^{-1}q) \cdot q} \; \hbar^{-n} \int_{\R^{2n}} e^{-\frac{1}{\hbar}\left(\frac{-\mathcal{B}(y,y)}{2} - \sum_{g \geq 0} \hbar^g \mathcal{F}_g(q' + D^{-1}q)\right)}dy.
\]
Each of the terms $-\mathcal{F}_g(q' + D^{-1}q)$ can be decomposed into pieces that are homogeneous in $q'$:
\[
-\mathcal{F}_g(q' + D^{-1}q) = \sum_{m \geq 0} \frac{1}{m!}\left(- \d^m\left.\mathcal{F}_g\right|_{q' = D^{-1}q} \cdot (q')^m\right),
\]
where $ -\d^m\mathcal{F}_g|_{q' = -D^{-1}q} \cdot (q')^m$ is short-hand for the $m$-tensor
\[
B_{g,m} = -\sum_{|\mathbf{m}| = m} \frac{m!}{m_1! \cdots m_n!} \left.\frac{\d^m \mathcal{F}_g}{(\d q_1')^{m_1} \cdots (\d q_n')^{m_n}}\right|_{q' = D^{-1}q}  (q_1')^{m_1} \cdots (q_n')^{m_n},
\]
in which the sum is over all $n$-tuples $\mathbf{m} = (m_1, \ldots, m_n) \in \Z_{\geq 0}^n$ such that $|\mathbf{m}| = m_1 + \cdots + m_n = m$.  We consider this as an $m$-tensor in $q'$ whose coefficients involve a formal parameter $q$.

Thus, we have expressed the quantized operator as
\[
(U_T\psi)(q) =  e^{\frac{1}{2\hbar}(BD^{-1}q) \cdot q} \; \hbar^{-n} \int_{\R^{2n}} e^{-\frac{1}{\hbar}\left( \frac{-\mathcal{B}(y,y)}{2} + \sum_{g,m \geq 0} \frac{\hbar^g}{m!} B_{g,m}(q', \ldots, q')\right)}dy.
\]
This is essentially the setting in which Feynman's theorem applies, but we  must modify Feynman's theorem to allow for the presence of powers of $\hbar$ in the exponent. This is straightforward, but nevertheless interesting, as it introduces a striking interpretation of $g$ as recording the ``genus" of vertices in a graph.

\begin{thm}
\label{modFeynman}
Let $V$ be a vector space of dimension $d$, and let
\[
S(x) = \frac{1}{2}B(x,x) + \sum_{g,m \geq 0} \frac{\hbar^g}{m!}B_{g,m}(x, \ldots, x),
\]
in which each $B_{g,m}$ is an $m$-multilinear form and $B_{0,0} = B_{0,1} = B_{0,2} = 0$.  Consider the integral
\[
Z = \hbar^{-\frac{d}{2}} \int_V e^{-S(x)/\hbar} dx.
\]
Then
\[
Z = \frac{(2\pi)^{d/2}}{\sqrt{\det(B)}} \sum_{ \mathbf{n} = (n_0, n_1, \ldots )} \sum_{\Gamma \in G'(\mathbf{n})} \frac{\hbar^{-\chi_{\Gamma}}}{|\text{Aut}(\Gamma)|} F_{\Gamma},
\]
where $F_{\Gamma}$ is the genus-modified Feynman amplitude, defined below.
\end{thm}

Here, $G'(\mathbf{n})$ denotes the set of isomorphism classes of graphs with $n_i$ vertices of valence $i$ for each $i \geq 0$, in which each vertex $v$ is labeled with a genus $g(v) \geq 0$.  Given $\Gamma \in G'(\mathbf{n})$, the {\bf genus-modified Feynman amplitude} is defined by the following procedure:
\begin{enumerate}
\item Put the $m$-tensor $-B_{g,m}$ at each $m$-valent vertex of genus $g$ in $\Gamma$.
\item For each edge $e$ of $\Gamma$, take the contraction of tensors attached to the vertices of $e$ using the bilinear form $B^{-1}$.  This will produce a number $F_{\Gamma_i}$ for each connected component $\Gamma_i$ of $\Gamma$.
\item If $\Gamma = \bigsqcup_{i} \Gamma_i$ is the decomposition of $\Gamma$ into connected components, define $F_{\Gamma} = \prod_i F_{\Gamma_i}$.
\end{enumerate}
Furthermore, the {\bf Euler characteristic} of $\Gamma$ is defined as
\[\chi_{\Gamma} =- \sum_{v \in V(\Gamma)} g(v) + |V(\Gamma)| - |E(\Gamma)|.\]

Having established all the requisite notation, the proof of the theorem is actually easy.

\begin{proof}[Proof of Theorem \ref{modFeynman}]
Reiterate the proof of Feynman's theorem to obtain
\[
Z = \sum_{\mathbf{n}= (n_{g,m})_{g,m \geq 0}} \left(\prod_{g,m \geq 0} \frac{\hbar^{gn_{g,m} + (\frac{m}{2} -1)n_{g,m}}}{(m!)^{n_{g,m}} n_{g,m}!}\right) \int_V e^{-B(y,y)/2} \prod_{g,m \geq 0} (-B_{g,m})^{n_{g,m}}dy.
\]
As before, Wick's theorem shows that the integral contributes the desired summation over graphs, modulo factors coming from over-counting.  A similar orbit-stabilizer argument shows that these factors precisely cancel the factorials in the denominator.  The power of $\hbar$ is
\[
\sum_{g \geq 0} g\left(\sum_{m \geq 0} n_{g,m}\right) + \sum_{m \geq 0}\left(\frac{m}{2} -1\right) \left(\sum_{g\geq 0} n_{g,m}\right),
\]
which is the sum of the genera of the vertices plus the number of edges minus the number of vertices, or in other words, $-\chi_{\Gamma}$, as required.
\end{proof}

It should be noted that for the proof of the theorem, there is no particular reason to think of $g$ as recording the genus of a vertex---it is simply a label associated to the vertex that records which of the $m$-tensors $B_{g,m}$ one attaches.  The convenience of the interpretation of $g$ as genus comes only from the fact that it simplifies the power of $\hbar$ neatly.

\subsection{Feynman diagram formula for $U_T$}

We are now ready to give an expression for $U_T \psi$ in terms of Feynman diagrams.  In order to apply Theorem \ref{modFeynman}, we must make one more assumption: that $\mathcal{F}_0$ has no terms of homogeneous degree less than $3$ in $q'$.  Assuming this, we obtain the following expression by directly applying the theorem:
\begin{equation}
\label{Feynmanfinal}
(U_T\psi)(q) = \frac{(2\pi)^n e^{\frac{1}{2\hbar}(BD^{-1}q) \cdot q}}{\sqrt{\det(D^{-1}C)}} \sum_{\mathbf{n} = (n_0, n_1, \ldots )} \sum_{\Gamma \in G'(\mathbf{n})} \frac{\hbar^{-\chi_{\Gamma}}}{|\text{Aut}(\Gamma)|} F_{\Gamma}(q),
\end{equation}
where $F_{\Gamma}(q)$ is the genus-modified Feynman amplitude given by placing the $m$-tensor
\[
\sum_{|\mathbf{m}| = m} \frac{m!}{m_1! \cdots m_n!} \left.\frac{\d^m \mathcal{F}_g}{(\d q_1')^{m_1} \cdots (\d q_n')^{m_n}}\right|_{q' = D^{-1}q}  (q_1')^{m_1} \cdots (q_n')^{m_n}
\]
at each $m$-valent vertex of genus $g$ in $\Gamma$ and taking the contraction of tensors using the bilinear form
\[
\left(\begin{matrix} 
D^{-1}C & I\\ I & 0 
\end{matrix}\right)^{-1} 
= -\left(\begin{matrix}
 0 & I\\ I & D^{-1}C
 \end{matrix}\right).
\]
In fact, since this bilinear form is only ever applied to vectors of the form $(0, q')$, we are really only taking contraction of tensors using the bilinear form $-D^{-1}C$ on $\R^n$.

\subsection{Connected graphs}

Recall from Theorem \ref{connected} that the logarithm of a Feynman diagram sum yields the sum over only connected graphs.  That is:
\[
(U_T\psi)(q) = \frac{(2\pi)^n e^{\frac{1}{2\hbar}(BD^{-1}q)\cdot q}}{\sqrt{\det(D^{-1}C})} \exp\left( \sum_{\Gamma \in G'_c} \frac{\hbar^{-\chi_{\Gamma}}}{|\text{Aut}(\Gamma)|} F_{\Gamma}(q)\right),
\]
where $G'_c$ denotes the set of isomorphism classes of {\it connected} genus-labeled graphs.  Thus, writing
\[
\overline{\mathcal{F}}_g(q) = \sum_{\Gamma \in G'_c(g)} \frac{1}{|\text{Aut}(\Gamma)|}F_{\Gamma}(q)
\]
with $G'_c(g)$ collecting connected graphs of genus $g$, we have:
\[
(U_T \psi)(q) = \frac{(2\pi)^n e^{\frac{1}{2\hbar}(BD^{-1}q)\cdot q}}{\sqrt{\det(D^{-1}C})} \exp\left(\sum_{g \geq 0} \hbar^{g-1}\overline{\mathcal{F}}_g(q) \ \right).
\]
For those familiar with Gromov-Witten theory, this expression should be salient---we will return to it in Chapter~\ref{infinite} of the book.

\subsection{The lower-triangular case}\label{ltc}
At this point, we can return to a remark made previously (Remark ~\ref{lowert}), regarding the computation of $U_T \psi$ when $T$ is lower-triangular.  In that case, we may express the quantization formula obtained via the \CCR  as a graph sum in a rather different way.  We explain how to do this below, and show that we ultimately obtain the same graph sum as that which arises from applying Feynman's Theorem to the integral operator. 

Let
\[
T = \left( \begin{matrix} 
I & 0 \\ C & I 
\end{matrix}\right).
\]
We showed in Section \ref{CCR} that
\[
(U_T\psi)(q) = \exp\left(-\frac{ \hbar}{2} C^{\alpha\beta} \frac{\d}{\d q^{\alpha}} \frac{\d}{\d q^{\beta}}\right) \psi(q).
\]
Suppose that $\psi(q) = e^{\sum_{g \geq 0} \hbar^{g-1}\mathcal{F}_g(q)}$ as above and expand both exponentials in Taylor series.  Then $(U_T\psi)(q)$ can be expressed as:
\begin{equation}
\label{UTpsi}
\sum_{\{i_{\alpha,\beta}\}, \{\ell_g\}} \frac{\hbar^{\sum i_{\alpha\beta} + \sum \ell_g(g-1)}}{\prod i_{\alpha\beta}!\prod \ell_g!} \prod_{\alpha,\beta}\left(-\frac{C^{\alpha\beta}}{2}\right)^{i_{\alpha\beta}} \left(\frac{\d}{\d q^{\alpha}} \frac{\d}{\d q^{\beta}}\right)^{i_{\alpha\beta}} \prod_g(\mathcal{F}_g(q))^{\ell_g}.
\end{equation}

Whenever a product of quadratic differential operators acts on a product of functions, the result can be written as a sum over diagrams.  As an easy example, suppose one wishes to compute
\[
\frac{\d^2}{\d x \d y}(fgh)
\]
for functions $f,g,h$ in variables $x$ and $y$.  The product rule gives nine terms, each of which can be viewed as a way of attaching an edge labeled

\noindent
\tikzstyle vertex=[circle, draw, fill=black!50,
inner sep=0pt, minimum width=4pt]
\begin{minipage}{\linewidth}
\centering
\begin{tikzpicture}[thick,scale=1]
\draw \foreach \x in {0}
{
(\x, 0) node[vertex] {} -- (\x+2,0) node[vertex] {}
};
\foreach \y in {0}
{
\node[above] at (\y,0){$\frac{\d}{\d x}$};
}
\foreach \y in {2}
{
\node[above] at (\y,0){$\frac{\d}{\d y}$};
}
\end{tikzpicture}\quad
\end{minipage}

\vspace{0.1cm}

\noindent to a collection of vertices labeled $f, g,$ and $h$.  (We allow both ends of an edge to be attached to the same vertex.)

Applying this general principle to the expression (\ref{UTpsi}), one can write each of the products $\prod (-\frac{C^{\alpha\beta}}{2})^{i_{\alpha\beta}}\prod\left(\frac{\d}{\d q^{\alpha}} \frac{\d }{\d q^{\beta}}\right)^{i_{\alpha\beta}} \mathcal{F}_g(q)$ as a sum over graphs obtained by taking $\ell_g$ vertices of genus $g$ for each $g$, with vertices of genus $g$ labeled $\mathcal{F}_g$, and attaching $i_{\alpha\beta}$ edges labeled
\tikzstyle vertex=[circle, draw, fill=black!50,
inner sep=0pt, minimum width=4pt]

\noindent
\begin{minipage}{\linewidth}
\centering
\begin{tikzpicture}[thick,scale=1]
\draw \foreach \x in {0}
{
(\x, 0) node[vertex] {} -- (\x+3,0) node[vertex] {}
};
\foreach \y in {0}
{
\node[above] at (\y,0){$\frac{\d}{\d x}$};
}
\foreach \y in {1.5}
{
\node[above] at (\y,0){$-C^{\alpha\beta}/2$};
}
\foreach \y in {3}
{
\node[above] at (\y,0){$\frac{\d}{\d y}$};
}
\end{tikzpicture}\quad
\end{minipage}
in all possible ways.  Each possibility gives a graph $\hat{Gamma}$.  It is a combinatorial exercise to check that the contributions from all choices of $\hat{\Gamma}$ combine to give a factor of $\frac{\prod i_{\alpha\beta}! \cdot \prod \ell_g!}{|\text{Aut}(\hat{\Gamma)}|}$.

Thus, we have expressed $(U_T\psi)(q)$ as
\[
\sum_{\{i_{\alpha\beta}\}, \{\ell_g\}, \hat{\Gamma}} \frac{\hbar^{-\chi_{\hat{\Gamma}}}}{|\text{Aut}(\hat{\Gamma})|} G_{\mathbf{i}, \mathbf{\ell}, \hat{\Gamma}}(q),
\]
where $G_{\mathbf{i}, \mathbf{\ell},\hat{\Gamma}}(q)$ is obtained by way of the above procedure.  This is essentially the Feynman amplitude computed previously, but there is one difference: $G_{\mathbf{i}, \mathbf{\ell}, \Gamma}(q)$ is computed via edges whose two ends are {\it labeled}, whereas Feynman diagrams are unlabeled.  By summing up all possible labelings of the same unlabeled graph $\Gamma$, we can rewrite this as
\[
(U_T\psi)(q) = \sum_{\Gamma} \frac{\hbar^{-\chi_{\Gamma}}}{|\text{Aut}(\Gamma)|} G_{\Gamma},
\]
where $G_{\Gamma}$ is the Feynman amplitude computed by placing the $m$-tensor
\[
\sum_{|\mathbf{m}| = m} \frac{m!}{m_1! \cdots m_n!} \left.\frac{\d^m \mathcal{F}_g}{(\d q_1')^{m_1} \cdots (\d q_n')^{m_n}}\right|_{q' = q}  (q_1')^{m_1} \cdots (q_n')^{m_n}
\]
at each $m$-valent vertex of genus $g$ in $\Gamma$ and taking the contraction of tensors using the bilinear form $-\frac{1}{2}(C + C^T) = - C$.  Up to a multiplicative constant, which we can ignore, this matches the Feynman diagram expansion obtained previously.

\section{Generalizations}
\label{generalizations}

As remarked in Section \ref{integrals}, one advantage of the integral formula representation of a quantized operator is that it generalizes in at least two ways beyond the cases considered here.

First, if $T$ is a symplectic diffeomorphism that is not necessarily linear, one can still define $\phi$ as the generating function of the graph of $T$, and under this definition, the integral in \eqref{intformula} still makes sense.  Thus, in principle, integral formulas allow one to define the quantization of an arbitrary symplectic diffeomorphism.  As it turns out, the formula in (\ref{intformula}) is no longer quite right in this more general setting; the constant $\lambda$ should be allowed to be a function $b = b(q, p', \hbar)$ determined by $T$ and its derivatives.  Nevertheless, an integral formula can still be obtained.  This is very important from a physical perspective, since the state space of classical mechanics is typically a nontrivial symplectic manifold.  While one can reduce to the case of $\R^{2n}$ by working locally, the quantization procedure should be functorial with respect to arbitrary symplectic diffeomorphisms, which can certainly be nonlinear even in local coordinates.

A second possible direction for generalization is that the function $\phi$ need not be the generating function of the graph of a symplectomorphism at all.  Any Lagrangian submanifold $L \subset T^*(\R^{2n})$ has a generating function, and taking $\phi$ to be the generating function of this submanifold, (\ref{intformula}) gives a formula for the quantization of  $L$.

\chapter{Interlude: Basics of Gromov-Witten theory}

In order to apply the formulas for $U_T\psi$  to obtain results in Gromov-Witten theory, it is necessary to quantize infinite-dimensional symplectic vector spaces.  Thus, we devote Chapter \ref{infinite} to the infinite-dimensional situation, discussing how to adapt the finite-dimensional formulas and how to avoid issues of convergence.  Before doing this, however, we pause to give a brief overview of the basics of Gromov-Witten theory.  Although this material will not be strictly necessary until Chapter 4, we include it now to motivate our interest in the infinite-dimensional case and the specific assumptions made in the next chapter.

\section{Definitions}

The material of this section can be found in any standard reference on Gromov-Witten theory, for example \cite{CoxKatz} or \cite{Hori}.

Let $X$ be a projective variety.  Roughly speaking, the Gromov-Witten invariants of $X$ encode the number of curves passing through a prescribed collection of subvarieties.  In order to define these invariants rigorously, we will first need to define the moduli space $\M_{g,n}(X, d)$ of stable maps.

\begin{df}
A genus-$g$, $n$-pointed {\bf pre-stable curve} is an algebraic curve $C$ with $h^1(C, \O_C) = g$ and at worst nodal singularities, equipped with a choice of $n$ distinct ordered marked points $x_1, \ldots, x_n \in C$.
\end{df}

Fix non-negative integers $g$ and $n$, and a cycle $d \in H_2(X; \Z)$.

\begin{df}
A {\bf pre-stable map} of genus $g$ and degree $d$ with $n$ marked points is an algebraic map $f: C \rightarrow X$ whose domain is a genus-$g$, $n$-pointed pre-stable curve, and for which $f_*[C] = d$.  Such a map is {\bf stable} if it has only finitely many automorphisms as a pointed map; concretely, this means that every irreducible component of genus zero has at least three special points (marked points or nodes) and every irreducible component of genus one has at least one special point.
\end{df}

As eluded to in this definition, there is a suitable notion of isomorphism of stable maps.  Namely, stable maps $f: C \rightarrow X$ and $f': C' \rightarrow X$ are {\bf isomorphic} if there is an isomorphism of curves $s: C \rightarrow C'$ which preserves the markings and satisfies $f' \circ s = f$.

There is a moduli space $\M_{g,n}(X, d)$ whose points are in bijection with isomorphism classes of stable maps of genus $g$ and degree $d$ with $n$ marked points.  To be more precise, $\M_{g,n}(X, d)$ is only a coarse moduli scheme, but it can be given the structure of a Deligne-Mumford stack, and with this extra structure it is a fine moduli stack.  In general, this moduli space is singular and may possess components of different dimensions.  However, there is always an ``expected" or ``virtual" dimension, denoted vdim,
and a class $[\M_{g,n}(X,d)]^{\text{vir}} \in H_{2\text{vdim}}(\M_{g,n}(X,d))$ , called the {\bf virtual fundamental class}, which plays the role of the fundamental class for the purpose of intersection theory.  The following theorem collects some of the important (and highly non-trivial) properties of this moduli space.  

\begin{thm}
There exists a compact moduli space $\M_{g,n}(X, d)$, of virtual dimension equal to
\[\text{vdim} = (\dim(X) - 3)(1-g) + \int_d c_1(T_X) + n.\]
It admits a virtual fundamental class $[\M_{g,n}(X,d)]^{\text{vir}} \in H_{2\text{vdim}}(\M_{g,n}(X,d))$.
\end{thm}

Expected dimension can be given a precise meaning in terms of deformation theory, which we will omit.  In certain easy cases, though, the moduli space is smooth and pure-dimensional, and in these cases the expected dimension is simply the ordinary dimension, and the virtual fundamental class is the ordinary fundamental class.  For example, this occurs when $g=0$ and $X$ is convex (such as the case $X = \P^r$) or when $g=0$ and $d=0$.

Gromov-Witten invariants will be defined as integrals over the moduli space of stable maps.  The classes we will integrate will come from two places.  First, there are {\bf evaluation maps}
\[\text{ev}_i: \M_{g,n}(X, d) \rightarrow X\]
for $i=1, \ldots, n$, defined by sending $(C; x_1, \ldots, x_n; f)$ to $f(x_i)$.  In fact, these set-theoretic maps are morphisms of schemes (or stacks).  Second, there are {\bf $\psi$ classes}.  To define these, let $\mathcal{L}_i$ be the line bundle on $\M_{g,n}(X, d)$ whose fiber over a point $(C; x_1, \ldots, x_n; f)$ is the cotangent line\footnote{Of course, this is only a heuristic definition, as one cannot specify a line bundle by prescribing its fibers.  To be more precise, one must consider the universal curve $\pi: \mathcal{C} \rightarrow \M_{g,n}(X,d)$.  This carries a relative cotangent line bundle $\omega_{\pi}$.  Furthermore, there are sections $s_i: \M_{g,n}(X, d) \rightarrow \mathcal{C}$ sending $(C; x_1, \ldots, x_n; f)$ to $x_i \in C \subset \mathcal{C}$.  We define $\mathcal{L}_i = s_i^*\omega_{\pi}$.} to the curve $C$ at $x_i$.  Then
\[\psi_i = c_1(\mathcal{L}_i)\]
for $i=1, \ldots, n$.

\begin{df}
Fix cohomology classes $\gamma_1, \ldots, \gamma_n \in H^*(X)$ and integers $a_1, \ldots, a_n \in \Z_{\geq 0}$.  The  corresponding {\bf Gromov-Witten invariant} (or {\bf correlator}) is
\[\langle \tau_{a_1}(\gamma_1), \cdots, \tau_{a_n}(\gamma_n) \rangle_{g,n,d}^X = \int_{[\M_{g,n}(X, d)]^{\textit{vir}}} ev_1^*(\gamma_1)\psi_1^{a_1} \cdots ev_n^*(\gamma_n) \psi_n^{a_n}.\]
\end{df}

While the enumerative significance of this integral is not immediately obvious, there is an interpretation in terms of curve-counting in simple cases.  Indeed, suppose that $\M_{g,n}(X, d)$ is smooth and its virtual fundamental class is equal to its ordinary fundamental class.  Suppose, further, that the classes $\gamma_i$ are Poincar\'e dual to transverse subvarieties $Y_i \subset X$, and that $a_i = 0$ for all $i$.  Then the Gromov-Witten invariant above is equal to the number of genus-$g$, $n$-pointed curves in $X$ whose first marked point lies on $Y_1$, whose second marked point lies on $Y_2$, et cetera.  Thus, the invariant indeed represents (in some sense) a count of the number of curves passing thorugh prescribed subvarieties.

In order to encode these invariants in a notationally parsimonious way, we write
\[\mathbf{a}^i(z) = a^i_0 + a^i_1z + a^i_2z^2 + \cdots\]
for $a^i_j \in H^*(X)$.  Then, given $\mathbf{a}^1, \ldots, \mathbf{a}^n \in H^*(X)[[z]]$, define
\[\langle \mathbf{a}^1(\psi), \ldots, \mathbf{a}^n(\psi)\rangle_{g,n,d}^X = \int_{[\M_{g,n}(X, d)]^{\text{vir}}} \left(\sum_{j=0}^{\infty} ev_1^*(a_j^1)\psi_1^j\right) \cdots \left(\sum_{j=0}^{\infty} ev_n^*(a_j^n)\psi_n^j\right).\]

Equipped with this notation, we can describe all of the genus-$g$ Gromov-Witten invariants of $X$ in terms of a generating function
\[\mathcal{F}^g_X(\mathbf{t}(z)) = \sum_{n,d \geq 0} \frac{Q^d}{n!} \langle \mathbf{t}(\psi), \ldots, \mathbf{t}(\psi)\rangle^X_{g,n,d},\]
which is a formal function of $\mathbf{t}(z) \in H^*(X)[[z]]$ taking values in the {\bf Novikov ring} $\C[[Q]]$.  Introducing another parameter $\hbar$ to record the genus, we can combine all of the genus-$g$ generating functions into one:
\[\mathcal{D}_X= \exp\left(\sum_g \hbar^{g-1} \mathcal{F}_X^g\right).\]
This is referred to as the {\bf total descendent potential} of $X$.

It is convenient to sum invariants with certain fixed insertions over all ways of adding additional insertions, as well as all choices of degree.  Thus, we define:
\[\langle \langle \mathbf{a}_1(\psi), \ldots, \mathbf{a}_n(\psi) \rangle \rangle_{g,m}^X(\tau) = \sum_{n,d} \frac{Q^d}{n!} \langle \mathbf{a}_1(\psi), \ldots, \mathbf{a}_n(\psi), s, \ldots, s \rangle_{g,n+m,d}^X\]
for specified $s \in H^*(X)$.

\section{Basic Equations}

In practice, Gromov-Witten invariants are usually very difficult to compute by hand.  Instead, calculations are typically carried out combinatorially by beginning from a few easy cases and applying a number of relations.  We state those relations in this section.  The statements of the relations can be found, for example, in \cite{Hori}, while their expressions as differential equations are given in \cite{Lee}.

\subsection{String Equation}

The string equation addresses invariants in which one of the insertions is $1 \in H^0(X)$ with no $\psi$ classes.  It states:
\begin{equation}\label{str1}\langle \tau_{a_1}(\gamma_1), \cdots ,\tau_{a_n}(\psi_n)\; 1 \rangle_{g,n+1,d}^X = \hspace{6cm} \end{equation}
\[\hspace{2cm}\sum_{i=1}^n\langle \tau_{a_1}(\gamma_1), \cdots, \tau_{a_{i-1}}(\gamma_{i-1}), \tau_{a_i -1}(\gamma_i), \tau_{a_{i+1}}(\gamma_{i+1}), \cdots, \tau_{a_n}(\gamma_n)\rangle_{g,n,d}^X\]
whenever $\M_{g,n}(X, d)$ is nonempty.

The proof of this equation relies on a result about pullbacks of $\psi$ classes under the forgetful morphism $\M_{g,n+1}(X, d) \rightarrow \M_{g,n}(X,d)$ that drops the last marked point.  Because this morphism involves contracting irreducible components of $C$ that becoming unstable after the forgetting operation, $\psi$ classes on the target do not pull back to $\psi$ classes on the source.  However, there is an explicit comparison result, and this is the key ingredient in the proof of \eqref{str1}.  See \cite{Hori} for the details.

It is a basic combinatorial fact that differentiation of the generating function with respect to the variable $t_j^i$ corresponds to adding an additional insertion of $\tau_j(\phi_i)$.  Starting from this, one may express the string equation in terms of a differential equation satisfied by the Gromov-Witten generating function.  To see this, fix a basis $\{\phi_1, \ldots, \phi_k\}$ for $H^*(X)$, and write
\[t_j = t_j^i \phi_i.\]
Then
\begin{align*}
\frac{\d }{\d t^1_0}\sum_g \hbar^{g-1}\mathcal{F}^g_X & = \sum_{g,n,d} \frac{Q^d\hbar^{g-1}}{n!} \left\langle \mathbf{t}(\psi), \ldots, \mathbf{t}(\psi), 1 \right\rangle^X_{g,n+1,d}\\
&= \sum_{g,n,d} \frac{Q^d\hbar^{g-1}}{(n-1)!} \left\langle \mathbf{t}(\psi), \ldots, \mathbf{t}(\psi), \sum_{i,j} t_{j+1}^i \tau_j(\phi_i )\right\rangle_{g,n,d}^X\\
&\hspace{2cm}+\frac{1}{2\hbar}\langle \mathbf{t}(\psi), \mathbf{t}(\psi), 1\rangle_{0,3,0}^X+ \langle 1 \rangle_{1,1,0} \\
&= \sum_{i,j} t_{j+1}^i \frac{\d }{\d t^i_j}\sum_g \hbar^{g-1} \mathcal{F}^g_X +\frac{1}{2\hbar}\langle \mathbf{t}(\psi), \mathbf{t}(\psi), 1\rangle_{0,3,0}^X+ \langle 1 \rangle_{1,1,0}. \\
\end{align*}
The ``exceptional" terms at the end arise from reindexing the summation, because the moduli spaces $\M_{g,n}(X,d)$ do not exist for $(g,n,d) = (0,2,0)$ or $(1,0,0)$.  The first  exceptional term is equal to
\[\frac{1}{2\hbar} \langle t_0, t_0\rangle^X,\]
where $\langle \; , \; \rangle^X$ denotes the Poincar\'e pairing on $X$, because the $\psi$ classes are trivial on $\M_{0,3}(X, 0)$.  The second exceptional term vanishes for dimension reasons.

Taking the coefficient of $\hbar^{-1}$, we find that $\mathcal{F}^0_X$ satisfies the following differential equation:
\begin{equation}
\label{SE}
\frac{\d \mathcal{F}^0_X}{\d t_0^1} = \frac{1}{2} \langle t_0, t_0 \rangle^X + \sum_{i,j} t^i_{j+1} \frac{\d \mathcal{F}^0_X}{\d t^i_j}.
\end{equation}

Before we continue, let us remark that the total-genus string equation can be presented in the following alternative form:
\begin{equation*}
\sum_{g,n,d} \frac{Q^d\hbar^{g-1}}{(n-1)!}\langle 1, \mathbf{t}(\psi), \ldots, \mathbf{t}(\psi)\rangle_{g,n,d}^X = \sum_{g,n,d}  \frac{Q^d\hbar^{g-1}}{(n-1)!}\left\langle \left[ \frac{\mathbf{t}(\psi)}{\psi}\right]_+ \hspace{-0.25cm},\mathbf{t}(\psi), \ldots, \mathbf{t}(\psi) \right\rangle_{g,n,d}^X
\end{equation*}
\begin{equation}
\label{trunc}
\hspace{6cm}+\frac{1}{2\hbar}\langle t_0, t_0 \rangle^X.
\end{equation}
This expression will be useful later.

\subsection{Dilaton Equation}

The dilaton equation addresses the situation in which there is an insertion of $1 \in H^*(X)$ with a first power of $\psi$ attached to it:
\[\langle \tau_{a_1}(\gamma_1), \cdots, \tau_{a_n} \tau_1(T_0) \rangle_{g,n,d}^X = (2g-2+n)\langle \tau_{a_1}(\gamma_1), \ldots, \tau_{a_n}(\gamma_n)\rangle_{g,n,d}^X.\]
Again, it can be expressed as a differential equation on the generating function.  In genus zero, the equation is:
\begin{equation}
\label{DE}
\frac{\d \mathcal{F}^0_X}{\d t_1^1} = \sum_{\substack{1 \leq i \leq k\\ j \geq 0}} t^i_j \frac{\d \mathcal{F}^0_X}{\d t_j^i} - 2\mathcal{F}^0_X.
\end{equation}
The proof is similar to the above, so we omit it.

A simple but extremely important device known as the {\bf dilaton shift} allows us to express this equation in a simpler form.  Define a new parameter $\mathbf{q}(z) = q_0 + q_1z + \cdots \in H^*(X)[[z]]$ by
\begin{equation}
\label{dilatonshift}
\mathbf{q}(z) = \mathbf{t}(z) - z,
\end{equation}
so that $q_i = t_i$ for $i \neq 1$ and $q_1 = t_1 - 1$.  If we perform this change of variables, then the dilaton equation says precisely that 
\[\sum_{\substack{1 \leq i \leq k\\ j \geq 0}} q^i_j \frac{\d \mathcal{F}^0_X}{\d q^i_j} = 2 \mathcal{F}^0_X,\]
or in other words that $\mathcal{F}^0_X(\mathbf{q}(z))$ is {\it homogeneous} of degree two.

\subsection{Topological Recursion Relations}

A more general equation relating Gromov-Witten invariants to ones with lower powers of $\psi$ is given by the topological recursion relations.  In genus zero, the relation is:
\begin{align*}
\langle \langle \tau_{a_1 + 1} (\gamma_1), \tau_{a_2}(\gamma_2), &\tau_{a_3}(\gamma_3)\rangle \rangle_{0,3}^X(\tau) \\
&=\sum_a \langle \langle \tau_{a_1}(\gamma_1), \phi_a \rangle \rangle_{0,2}^X \langle \langle \phi^a, \tau_{a_2}(\gamma_2), \tau_{a_3}(\gamma_3)\rangle \rangle_{0,3}^X,
\end{align*}
where, as above, $\{\phi_a\}$ is a basis for $H^*(X)$, and $\{\phi^a\}$ denotes the dual basis under the Poincar\'e pairing.  There are also topological recursion relations in higher genus (see \cite{EguchiXiong}, \cite{Getzler1}, and \cite{Getzler2}), but we omit them here as they are more complicated and not necessary for our purposes.

In terms of a differential equation, the genus-zero topological recursion relations are given by
\begin{equation}
\label{TRR}
\frac{\d^3 \mathcal{F}^0_X}{\d t_{j_1}^{i_1} \d t_{j_2}^{i_2} \d t_{j_3}^{i_3}} = \sum_{\mu, \nu} \frac{\d^2 \mathcal{F}^0_X}{\d t^{i_1}_{j_1} \d t^{\mu}_0} g^{\mu \nu} \frac{\d \mathcal{F}^0_X}{\d t^{i_2}_{j_2} \d t^{i_3}_{j_3} \d t^{\nu}_0}.
\end{equation}
Here, we use $g_{\mu\nu}$ to denote the matrix for the Poincar\'e pairing on $H^*(X)$ in the basis $\{\phi_{\alpha}\}$ and $g^{\mu \nu}$ to denote the inverse matrix.

\subsection{Divisor Equation}

The divisor equation describes invariants in which one insertion lies in $H^2(X)$ (with no $\psi$ classes) in terms of invariants with fewer insertions.  For $\rho \in H^2(X)$, it states:
\[\langle \tau_{a_1}(\gamma_1), \cdots ,\tau_{a_n}(\gamma_n) \; \rho \rangle^X_{g,n+1,d} = \langle \rho, d \rangle \langle \tau_{a_1}(\gamma_1), \cdots, \tau_{a_n}(\gamma_n)\rangle^X_{g,n,d}\hspace{3cm}\]
\[\hspace{2cm} + \sum_{i=1}^n \langle \tau_{a_1}(\gamma_1), \cdots, \tau_{a_{i-1}}(\gamma_{i-1}), \tau_{a_i -1}(\gamma_i \rho), \tau_{a_{i+1}}(\gamma_{i+1}), \cdots, \tau_{a_n}(\gamma_n)\rangle^X_{g,n,d},\]
or equivalently,
\begin{equation}
\label{divisoreqn}
\langle \mathbf{a}_1(\psi), \ldots, \mathbf{a}_{n-1}(\psi), \rho \rangle_{g,n,d}^X = \langle \rho, d \rangle \langle \mathbf{a}_1(\psi), \ldots, \mathbf{a}_{n-1}(\psi)\rangle_{g,n-1,d}^X
\end{equation}
\begin{equation*}
\hspace{2cm} + \sum_{i=1}^{n-1} \left\langle \mathbf{a}_1(\psi), \ldots, \left[\frac{\rho \mathbf{a}_i(\psi)}{\psi}\right]_+, \ldots, \mathbf{a}_{n-1}(\psi)\right\rangle_{g,n-1,d}^X.
\end{equation*}
This equation, too, can be expressed as a differential equation on the generating function.  The resulting equation is not needed for the time being, but we will return to it in Chapter 4.

\section{Axiomatization}
\label{axiomatic}

Axiomatic Gromov-Witten theory attempts to formalize the structures which arise in a genus-zero Gromov-Witten theory.  One advantage of such a program is that any properties proved in the framework of axiomatic Gromov-Witten theory will necessarily hold for any of the variants of Gromov-Witten theory that share the same basic properties, such as the orbifold theory or FJRW theory.  See \cite{Lee}, for a more detailed exposition of the subject of axiomatization.

Let $H$ be an arbitrary $\Q$-vector space equipped with a distinguished element $1$ and a nondegenerate inner product $(\; , \;)$.  Let
\[\mathbb{H} = H((z^{-1})),\]
which is a symplectic vector space with symplectic form $\Omega$ defined by
\[\Omega(f,g) = \text{Res}_{z=0}\bigg((f(-z), g(z))\bigg).\]
An arbitrary element of $\mathbb{H}$ can be expressed as
\begin{equation}
\label{darboux}
\sum_{k \geq 0} p_{k,\alpha}\phi^{\alpha}(-z)^{-1-k} + \sum_{\ell \geq 0} q_{\ell}^{\beta} \phi_{\beta} z^{\ell},
\end{equation}
in which $\{\phi_1, \ldots, \phi_d\}$ is a basis for $H$ with $\phi_1 = 1$.  Define a subspace $\mathbb{H}_+ = H[[z]]$ of $\mathbb{H}$, which has coordinates $q^i_j$.  Elements of $\mathbb{H}_+$ are identified with $\mathbf{t}(z) \in H[[z]]$ via the dilaton shift
\[t^i_j = q^i_j + \delta^{i1}\delta_{j1}.\]

\begin{df}
A {\bf genus-zero axiomatic theory} is a pair $(\mathbb{H}, G_0)$, where $\mathbb{H}$ is as above and $G_0= G_0(\mathbf{t})$ is a formal function of $\mathbf{t}(z) \in H[[z]]$ satisfying the differential equations (\ref{SE}), (\ref{DE}), and (\ref{TRR}).
\end{df}

In the case where $H = H^*(X ; \Lambda)$ equipped with the Poincar\'e pairing and $G_0 = \mathcal{F}^0_X$, one finds that the genus-zero Gromov-Witten theory of $X$ is an axiomatic theory.

Note that we require an axiomatic theory to satisfy neither the divisor equation (for example, this fails in orbifold Gromov-Witten theory), nor the WDVV equations, both of which are extremely useful in ordinary Gromov-Witten theory.  While these properties are computationally desirable for a theory, they are not necessary for the basic axiomatic framework.

\chapter{Quantization in the infinite-dimensional case} \label{infinite}

Axiomatization reduces the relevant structures of Gromov-Witten theory to a special type of function on an infinite-dimensional symplectic vector space
\begin{equation}
\label{H}
\mathbb{H} = H((z^{-1})).
\end{equation}
As we will see in Chapter \ref{GW}, the actions of quantized operators on the quantization of $\mathbb{H}$ have striking geometric interpretations in the case where $H = H^*(X; \Lambda)$ for a projective variety $X$.  Because our ultimate goal is the application of quantization to the symplectic vector space (\ref{H}), we will assume throughout this chapter that the infinite-dimensional symplectic vector space under consideration has that form.

\section{The symplectic vector space}

Let $H$ be a vector space of finite dimension equipped with a nondegenerate inner product $( \; , \; )$ and let $\mathbb{H}$ be given by (\ref{H}).  As explained in Section \ref{axiomatic}, $\mathbb{H}$ is a symplectic vector space under the symplectic form $\Omega$ defined by
\[\Omega(f,g) = \Res_{z=0} \bigg( (f(-z), g(z)) \bigg).\]
The subspaces
\[\mathbb{H}_+ = H[z]\]
and
\[\mathbb{H}_- = z^{-1}H[[z^{-1}]]\]
are Lagrangian. A choice of basis $\{\phi_1, \ldots, \phi_d\}$ for $H$ yields a symplectic basis for $\mathbb{H}$, in which the expression for an arbitray element in Darboux coordinates is
\[\sum_{k \geq 0} p_{k,\alpha} \phi^{\alpha} (-z)^{-1 - k} + \sum_{\ell \geq 0} q_{\ell}^{\beta}\phi_{\beta} z^{\ell}.\]
Here, as before, $\{\phi^{\alpha}\}$ denotes the dual basis to $\{\phi_{\alpha}\}$ under the pairing $(\; , \;)$. We can identify $\mathbb{H}$ as a symplectic vector space with the cotangent bundle $T^*\mathbb{H}_+$.

Suppose that $T: \mathbb{H} \rightarrow \mathbb{H}$ is an endomorphism of the form
\begin{equation}
\label{M}
T = \sum_{m} B_m z^m,
\end{equation}
where $B_m : H \rightarrow H$ are linear transformations.  Let $T^*$ denote the endomorphism given by taking the adjoint $B_m^*$ of each transformation $B_m$ with respect to the pairing.  Then the symplectic adjoint of $T$ is
\[T^{\dagger}(z) = T^*(-z) = \sum_m B_m^*(-z)^m.\]

As usual, a symplectomorphism is an endomorphism $T$ of $\mathbb{H}$ that satisfies $\Omega(Tf, Tg) = \Omega(f,g)$ for any $f,g\in \mathbb{H}$. When $T$ has the form $\eqref{M}$, one can check that this is equivalent to the condition 
\[
T^*(-z)T(z) = \Id.
\]

We will specifically be considering symplectomorphisms of the form
\[T = \exp(A)\]
in which $A$ also has the form $\eqref{M}$.  In this case, we will require that $A$ is an infinitesimal symplectic transformation, or in other words that $\Omega(Af, g) + \Omega(f, Ag) = 0$ for any $f,g \in \mathbb{H}$.  Using the expression for $A$ as a power series as in $\eqref{M}$, one finds that this condition is equivalent to $A^*(-z) + A(z) = 0$, which in turn implies that
\begin{equation}
\label{infsympl}
A^*_m=(-1)^{m+1}A_m.
\end{equation}

There is another important restriction we must make: the transformation $A$ (and hence $T$) will be assumed to contain {\it either} only nonnegative powers of $z$ {\it or} only nonpositive powers of $z$.  Because of the ordering of the coordinates $p_{k,\alpha}$ and $q^{\beta}_{\ell}$, a transformation
\[R = \sum_{m \geq 0} R_m z^m\]
with only nonnegative powers of $z$ will be referred to as lower-triangular, and a transformation
\[S = \sum_{m \leq 0} S_m z^m\]
with only nonpositive powers of $z$ will be referred to as upper-triangular.\footnote{This naming convention disagrees with what is typically used in Gromov-Witten theory, but we choose it for consistency with the finite-dimensional setting discussed in Chapter \ref{finite}.}  The reason for this is that if $A$ had both positive and negative powers of $z$, then exponentiating $A$ would yield a series in which a single power of $z$ could have a nonzero contribution from infinitely many terms, and the result would not obviously be convergent.  There are still convergence issues to be addressed when one composes lower-triangular with upper-triangular operators, but we defer discussion of this to Section \ref{convergence}.
 
\section{Quantization via quadratic Hamiltonians}

Although in the finite-dimensional case we needed to start by choosing a symplectic basis, by expressing our symplectic vector space as $\mathbb{H} = H((z^{-1}))$ we have already implicity chosen a polarization $\mathbb{H} = \mathbb{H}_+ \oplus \mathbb{H}_-$.  Thus, the quantization of the symplectic vector space $\mathbb{H}$ should be thought of as the Hilbert space $\mathscr{H}$ of square-integrable functions on $\mathbb{H}_+$ with values in $\C[[\hbar, \hbar^{-1}]]$.  As in the finite-dimensional case, we will sometimes in practice allow $\mathscr{H}$ to contain formal functions that are not square-integrable, and the space of all such formal functions will be referred to as the Fock space.

\subsection{The quantization procedure}
Observables, which classically are functions on $\mathbb{H}$, are quantized the same way as in the finite-dimensional case, by setting
\[
\hat{q}^{\alpha}_k = q^{\alpha}_k
\]
\[
\hat{p}_{k,\alpha} = \hbar \frac{\d}{\d q^{\alpha}_k},
\]
and quantizing an arbitrary analytic function by expanding it in a Taylor series and ordering the variables within each monomial in the form $q^{\alpha_1}_{k_1} \cdots q^{\alpha_n}_{k_n} p_{\ell_1,\beta_1} \cdots p_{\ell_m,\beta_{m}}$.

In order to quantize a symplectomorphism $T = \exp(A)$ of the form discussed in the previous section, we will mimic the procedure discussed in Section \ref{Ham}.  Namely, define a function $h_A$ on $\mathbb{H}$ by
\[h_A(f) = \frac{1}{2}\Omega(Af, f).\]
Since $h_A$ is a classical observable, it can be quantized by the above formula.  We define the quantization of $A$ via
\[
\hat{A} = \frac{1}{\hbar}\widehat{h_A},
\]
and $U_T$ is defined by
\[
U_T = \exp(\hat{A}).
\]
The next section is devoted to making this formula more explicit.

\subsection{Basic examples}
Before turning to the general upper-triangular and lower-triangular cases, we will begin with two simple but crucial examples.  These computations follow Example 1.3.3.1 of \cite{Coates}.

\begin{ex}
Suppose that the infinitesimal symplectic transformation $A$ is of the form
\[
A = A_m z^{m},
\]
where $A_m: H \rightarrow H$ is a linear transformation and $m> 0$.

To compute $\hat{A}$, one must first compute the Hamiltonian $h_A$.  Let
\[
f(z) = \sum_{k \geq 0} p_{k,\alpha} \phi^{\alpha} (-z)^{-1-k} + \sum_{\ell \geq 0} q_{\ell}^{\beta} \phi_{\beta} z^{\ell} \in \mathbb{H}.
\]
Then
\begin{align*}
h_A(f) &= \frac{1}{2} \Omega(Af, f)\\
&=\frac{1}{2}\Res_{z=0}\left((-z)^{m}\sum_{k_1 \geq 0} p_{k_1, \alpha}(A_m \phi^{\alpha})\ z^{-1-k_1} + (-z)^m\sum_{\ell_1 \geq 0} q_{\ell_1}^{\beta}(A_m \phi_{\beta})(-z)^{\ell_1},\right.\\
&\hspace{5cm} \left.\sum_{k_2 \geq 0} p_{k_2, \alpha} \phi^{\alpha} (-z)^{-1-k_2} + \sum_{\ell_2 \geq 0} q_{\ell_2}^{\beta} \phi_{\beta}z^{\ell_2}\right)
\end{align*}
Since only the $z^{-1}$ terms contribute to the residue, the right-hand side is equal to
\begin{align*}
\frac{1}{2}\sum_{k \geq 0}^{m-1} (-1)^k p_{k,\alpha} p_{m-k-1, \beta} (A_m \phi^{\alpha}, \phi^{\beta}) +& \frac{1}{2}\sum_{k \geq 0}(-1)^m p_{m+k, \alpha} q_{k}^{\beta}(A_m \phi^{\alpha}, \phi_{\beta})\hspace{1cm}\\
	&-\frac{1}{2}\sum_{k \geq 0} q_{k}^{\beta}p_{m+k, \alpha}(A_m \phi_{\beta}, \phi^{\alpha}).
\end{align*}
By (\ref{infsympl}), we have 
\[
(\phi^{\alpha}, A_m\phi_{\beta})=(-1)^{m+1}(A_m \phi^{\alpha}, \phi_{\beta}).
\]
Thus, denoting $(A_m)^{\alpha}_{\beta} = (A_m \phi^{\alpha}, \phi_{\beta})$ and $(A_m)^{\alpha \beta} = (A_m \phi^{\alpha}, \phi^{\beta})$, we can write
\[
h_A(f) = \frac{1}{2}\sum_{k=0}^{m-1} (-1)^k p_{k,\alpha} p_{m-k-1,\beta} (A_m)^{\alpha\beta}+(-1)^m\sum_{k \geq 0} p_{m+k, \alpha} q_{k}^{\beta}(A_m)^{\alpha}_{\beta}.
\]
This implies that
\[
\hat{A} = \frac{\hbar}{2} \sum_{k=0}^{m-1} (-1)^k (A_m)^{\alpha\beta} \frac{\d}{\d q^{\alpha}_k}\frac{\d}{\d q^{\beta}_{m-k-1}} + (-1)^m \sum_{k \geq 0} (A_m)^{\alpha}_{\beta} q^{\beta}_{k} \frac{\d}{\d q_{m+k}^{\alpha}}.
\]
\end{ex}

\begin{ex}
\label{Bzm}
Similarly to the previous example, let
\[
A = A_m z^{m},
\]
this time with $m < 0$.  Let  $(A_{m})_{ \alpha\beta} = (A_m\phi_{\alpha}, \phi_{\beta})$.  An analogous computation shows that
\[
\hat{A} = \frac{1}{2\hbar} \sum_{k=0}^{-m-1}(-1)^{k+1} (A_m)_{\alpha\beta} q_{-m-k-1}^{\alpha} q_{k}^{\beta}
+(-1)^m\sum_{k \geq -m}(A_m)^{\alpha}_{\beta}q^{\beta}_{k} \frac{\d}{\d q_{k+m}^{\alpha}}.
\]
\end{ex}

\begin{ex}
\label{lower-quad}
More generally, let 
\[
A=\sum_{m<0}A_m z^{m}
\] 
be an infinitesimal symplectic endomorphism.  Then the quadratic Hamiltonian associated to $A$ is 
\[
h_A(f)=\frac{1}{2}\sum_{k,m} (-1)^{m+1} (A_{-k-m-1})_{\alpha\beta} q_{k}^{\alpha}q_{m}^{\beta}
+\sum_{k,m}(-1)^{m} (A_m)^{\alpha}_{\beta}  p_{k+m,\alpha} q_{k}^{\beta}.
\]
Thus, after quantization, we obtain 
\begin{equation}
\hat{A}=\frac{1}{2\hbar}\sum_{k,m}(-1)^{m+1} (A_{-k-m-1})_{\alpha\beta} q_{k}^{\alpha}q_{m}^{\beta} 
+\sum_{k,m}(-1)^{m} (A_m)^{\alpha}_{\beta} q_{k}^{\beta}\frac{\d}{\d q_{k+m}^{\alpha}}.
\end{equation}
\end{ex}

It is worth noting that some of the fairly complicated expressions appearing in these formulas can be written more succinctly, as discussed at the end of Example 1.3.3.1 of \cite{Coates}.  Indeed, the expression
\begin{equation}
\label{dA}
\d_A := \sum_{k} (A_m)^{\alpha}_{\beta} q_k^{\beta} \frac{\d}{\d q^{\alpha}_{k+m}}
\end{equation}
that appears in $\hat{A}$ for both $m > 0$ and $m < 0$ acts on $\mathbf{q}(z) \in \mathbb{H}_+$ by
\[
\left(\sum_{k} (A_m)^{\alpha}_{\beta} q^{\beta}_k \frac{\d}{\d q^{\alpha}_{k+m}} \right)\left(\sum_{\ell} q_{\ell}^{\gamma} \phi_{\gamma} z^{\ell}\right) = \left[\sum_{k} (A_m)^{\alpha}_{\beta} q_k^{\beta} \phi_{\alpha} z^{k+m}\right]_+,
\]
where $[ \cdot ]_+$ denotes the power series truncation with only nonnegative powers of $z$.  In other words, if $\mathbf{q} = \sum_{\ell} q^i_{\ell}\phi_i z^{\ell}$, then
\begin{equation}
\label{qplus}
\d_A \mathbf{q} = [A\mathbf{q}]_+.
\end{equation}

By the same token, consider the expression
\[
\sum_{k \geq 0} (-1)^k (A_m)^{\alpha\beta} \frac{\d}{\d q_k^{\alpha}} \frac{\d}{\d q^{\beta}_{m-k-1}}
\]
appearing in $\hat{A}$ for $m>0$.  The quadratic differential operator $\frac{\d}{\d q_k^{\alpha}} \frac{\d}{\d q^{\beta}_{m-k-1}}$ can be thought of as a bivector field on $\mathbb{H}_+$, and since $\mathbb{H}_+$ is a vector space, a bivector field can be identified with a tensor product of two maps $\mathbb{H}_+ \rightarrow \mathbb{H}_+$.  Specifically, $\frac{\d}{\d q_k^{\alpha}} \frac{\d}{\d q^{\beta}_{m-k-1}}$ is the bivector field corresponding to the constant map
\[
\phi_{\alpha} z_+^k \otimes \phi_{\beta} z_-^{m-k-1} \in H[z_+] \otimes H[z_-] \cong \mathbb{H}_+ \otimes \mathbb{H}_+.
\]
Using the identity
\[
\sum_{k=0}^{m-1}(-1)^k z_+^k z_-^{m-1-k} = \frac{z_+^m +(-1)^{m-1} z_-^m}{z_+ + z_-},
\]
then, it follows that
\begin{equation}
\label{plusminus}
\sum_{k \geq 0} (-1)^k (A_m)^{\alpha\beta} \frac{\d}{\d q_k^{\alpha}} \frac{\d}{\d q^{\beta}_{m-k-1}} = \left[\frac{A(z_+) + A^*(z_-)}{z_+ + z_-}\right]_+,
\end{equation}
\noindent where we use the pairing to identify $(A_m)^{\alpha\beta}\phi_{\alpha}\otimes \phi_{\beta} \in \mathbb{H}_+ \otimes \mathbb{H}_+$ with $A_m \in \End(H)$.  Here, the power series truncation is included to ensure that (\ref{plusminus}) is trivially valid when $m$ is negative.

The modified expressions (\ref{qplus}) and (\ref{plusminus}) can be useful in recognizing the appearance of a quantized operator in computations.  For example, (\ref{qplus}) will come up in Proposition \ref{LT} below, and both (\ref{qplus}) and (\ref{plusminus}) arise in the context of Gromov-Witten theory in Theorem 1.6.4 of \cite{Coates}.

\subsection{Formulas for the lower-triangular and upper-triangular cases}

Extending the above two examples carefully, one obtains formulas for the quantization of any upper-triangular or lower-triangular symplectomorphism.  The following results are quoted from \cite{GiventalQuadraticHamiltonians}; see also \cite{LeePandh} for an exposition.

\begin{prop}
\label{LT}
Let $S$ be an upper-triangular symplectomorphism of $\mathbb{H}$ of the form $S = \exp(A)$, and write
\[
S(z) = I + S_1/z + S_2/z^2 + \cdots.
\]
Define a quadratic form $W_S$ on $\mathbb{H}_+$ by the equation
\[
W_S(\mathbf{q}) = \sum_{k,\ell\geq 0}(W_{k\ell}\mathbf{q}_{k}, \mathbf{q}_{\ell}),
\]
where 
$\mathbf{q}_k=q_k^{\alpha}\phi_{\alpha}$
and $W_{k\ell}$ is defined by
\[
\sum_{k,\ell \geq 0} \frac{W_{k\ell}}{z^k w^{\ell}} = \frac{S^*(w)S(z) - I}{w^{-1}+z^{-1}}.
\]
Then the quantization of $S^{-1}$ acts on the Fock space by
\[
(U_{S^{-1}}\Psi)(\mathbf{q}) = \exp\left(\frac{W_S(\mathbf{q})}{2\hbar}\right)\Psi([S\mathbf{q}]_+)
\]
for any function $\Psi$ of $\mathbf{q} \in \mathbb{H}_+$.  Here, as above, $[S\mathbf{q}]_+$ denotes the truncation of $S(z)\mathbf{q}$ to a power series in $z$.
\end{prop}

\begin{proof}
Let $A=\displaystyle\sum_{m<0} A_m z^m$.  Introduce a real parameter $t$ and denote
\[
G(t,\mathbf{q})=e^{-t\hat{A}}\Psi(\mathbf{q}).
\]
Define a $t$-dependent analogue of $W_S$ via
\begin{equation}
\label{p-sol}
W_{t}(\mathbf{q}):=\sum_{k,\ell\geq0}(W_{k\ell}(t)\mathbf{q}_{k},\mathbf{q}_{\ell}),
\end{equation}
where
\[
\sum_{k,\ell\geq0}\frac{W_{k\ell}(t)}{z^{k}w^{\ell}}=\frac{e^{t\cdot A^*(w)} e^{t\cdot A(z)}-1}{z^{-1}+w^{-1}}.
\]
Note that $W_{k\ell}(t) = W^*_{\ell k}(t)$.

We will prove that
\begin{equation}
\label{star}
G(t, \mathbf{q}) = \exp\left(\frac{W_t(\mathbf{q})}{2\hbar}\right) \psi\left(\left[e^{tA}\mathbf{q}\right]_+\right)
\end{equation}
for all $t$.  The claim follows by setting $t=1$.

To prove (\ref{star}), let
\[g(t, \mathbf{q}) = \log(G(t, \mathbf{q}))\]
and write $\Psi = \exp(f)$.  Then, taking logarithms, it suffices to show
\begin{equation}
\label{starstar}
g(t, \mathbf{q}) = \frac{W_t(\mathbf{q})}{2\hbar} + f\left(\left[e^{tA}\mathbf{q}\right]_+\right).
\end{equation}

Notice that
\begin{align*}
\frac{d}{dt}G(t,\mathbf{q})&=-\hat{A}\ G(t,\mathbf{q})\\
&=\frac{1}{2\hbar}\sum_{k,\ell}(-1)^{\ell}(A_{-k-\ell-1})_{\alpha\beta}q_{k}^{\alpha}q_{\ell}^{\beta} G(t,\mathbf{q})+\sum_{k,\ell}(-1)^{\ell-1}(A_\ell)^{\alpha}_{\beta}q_{k}^{\beta}\frac{\d}{\d q_{k+\ell}^{\alpha}} G(t,\mathbf{q}),
\end{align*}
using Example \eqref{lower-quad}.  This implies that $g(t, \mathbf{q})$ satisfies the differential equation
\begin{equation}\label{quad-pde}
\frac{d}{dt}g(t,\mathbf{q})=
\frac{1}{2\hbar}\sum_{k,\ell}(-1)^{\ell}A_{-k-\ell-1,\alpha\beta}q_{k}^{\alpha}q_{\ell}^{\beta}
+
\sum_{k,\ell}(-1)^{\ell-1}(A_\ell)^{\alpha}_{\beta} q_{k+\ell}^{\beta}
\frac{\d g}{\d q_{k}^{\alpha}}
\end{equation}
We will prove that the right-hand side of (\ref{starstar}) satisfies the same differential equation.

The definition of $W_{k\ell}(t)$ implies that
\[
\frac{d}{d t}W_{k \ell}(t)=
\sum_{\ell'=0}^{\ell}A^*_{\ell'-\ell}W_{k \ell'}(t)
+\sum_{k'=0}^{k}W_{k'\ell}(t)A_{k'-k}
+(-1)^{k} A_{-k-\ell-1}.
\]
Therefore,
\begin{align*}
\frac{1}{2\hbar}\frac{d}{dt}W_t(\mathbf{q}) &= \frac{1}{2\hbar}\sum_{k,\ell}\Bigg(\sum_{\ell'}\left(A_{\ell'-\ell}^*W_{k\ell'}(t)\mathbf{q}_{k},\mathbf{q}_{\ell}\right)+\sum_{k'}\left(W_{k'\ell}(t)A_{k'-k}\mathbf{q}_{k},\mathbf{q}_{\ell}\right)\\
    &\hspace{3cm} +(-1)^{k}(A_{-k-\ell-1}\mathbf{q}_{k},\mathbf{q}_{\ell})\Bigg)\\
   &= \frac{1}{2\hbar}\sum_{k,\ell}\left(2\sum_{\ell'}\left(W_{k\ell'}(t)\mathbf{q}_{k},A_{\ell'-\ell}\mathbf{q}_{\ell}\right)+(-1)^{k}\left(A_{-k-\ell-1}\mathbf{q}_{\ell},\mathbf{q}_{k}\right)\right)\\
   &= \frac{1}{2\hbar}\sum_{k,\ell}\bigg(2\sum_{\ell'}(-1)^{\ell'-\ell-1}(A_{\ell'-\ell})^{\alpha}_\beta q_{\ell}^\beta\left(W_{k\ell'}(t)\mathbf{q}_{k},\phi_\alpha\right)\\
   &\hspace{3cm}(-1)^{k}\left(A_{-k-\ell-1}\mathbf{q}_{\ell},\mathbf{q}_{k}\right)\bigg)\\
     &=\frac{1}{2\hbar}\sum_{k, \ell} (-1)^{\ell} (A_{-k-\ell-1})_{\alpha\beta}q^{\alpha}_{k} q^{\beta}_{\ell}\\
     &\hspace{3cm} +\frac{1}{\hbar} \sum_{k,\ell,\ell'} (-1)^{\ell'-\ell-1} (A_{\ell'-\ell})^{\alpha}_{\beta} q^{\beta}_{\ell} \left(W_{k\ell'}(t)\mathbf{q}_k, \frac{\d}{\d q^{\alpha}_{\ell'}} \mathbf{q}_{\ell'}\right)\\
   &=\frac{1}{2\hbar}\sum_{k, \ell} (-1)^{\ell} (A_{-k-\ell-1})_{\alpha\beta}q^{\alpha}_{k} q^{\beta}_{\ell}  + \sum_{k,\ell}(-1)^{\ell-1} (A_{\ell})^{\alpha}_{\beta} q^{\beta}_k \frac{\d}{\d^{\alpha}_{k+\ell}}\left(\frac{W_t(\mathbf{q})}{2\hbar}\right).
\end{align*}
Furthermore, using Equation (\ref{qplus}), it can be shown that
\[\frac{df\left(\left[e^{tA}\mathbf{q}\right]_+\right)}{dt} = \sum_{k,\ell} (-1)^{\ell-1} (A_{\ell})^{\alpha}_{\beta} q^{\beta}_k \frac{\d}{\d q^{\alpha}_{k+\ell}} f\left(\left[e^{tA}\mathbf{q}\right]_+\right).\]

Thus, both sides of (\ref{starstar}) satisfy the same differential equation.  Since they agree when $t=0$, and each monomial in $\mathbf{q}$ and $\hbar$ depends polynomially on $t$, it follows that the two sides of (\ref{starstar}) are equal.

\end{proof}

The lower-triangular case has an analogous proposition, but we omit the proof in this case.

\begin{prop}
\label{UT}
Let $R$ be a lower-triangular symplectomorphism of $\mathbb{H}$ of the form $R = \exp(B)$, and denote
\[R(z) = I + R_1z + R_2z^2 + \cdots.\]
Define a quadratic form $V_R$ on $\mathbb{H}_-$ by the equation
\[V_R(p_0(-z)^{-1} + p_1(-z)^{-2} + p_2(-z)^{-3} + \cdots ) = \sum_{k, \ell \geq 0} (p_k, V_{k\ell}p_{\ell}),\]
where $V_{k \ell}$ is defined by
\[\sum_{k, \ell \geq 0} (-1)^{k + \ell} V_{k \ell} w^k z^{\ell} = \frac{R^*(w) R(z) - I}{z + w}.\]
Then the quantization of $R$ acts on the Fock space by
\[(U_R \Psi)(\mathbf{q}) = \left[\exp\left(\frac{\hbar V_R(\d_{\mathbf{q}})}{2}\right)\Psi\right](R^{-1}\mathbf{q}),\]
where $V_R(\d_{\mathbf{q}})$ is the differential operator obtained from $V_R(\mathbf{p})$ by replacing $p_k$ by $\frac{\d}{\d q_k}$.
\end{prop}

\section{Convergence}
\label{convergence}
In Chapter 1, we expressed the quantization of an arbitrary symplectic transformation by decomposing it into a product of upper-triangular and lower-triangular transformations, each of whose quantizations was known.  In the infinite-dimensional setting, however, such a decomposition is problematic, because composing a series containing infinitely many nonnegative powers of $z$ with one containing infinitely many nonpositive powers will typically yield a divergent series.  This is why we have only defined quantization for upper-triangular or lower-triangular operators, not products thereof.

It is possible to avoid unwanted infinities if one is vigilant about each application of a quantized operator to an element of $\mathscr{H}$.  For example, while the symplectomorphism $S \circ R$ may not be defined, it is possible that $(\hat{S} \circ \hat{R})\Psi$ makes sense for a given $\Psi \in \mathscr{H}$ if $\hat{S} (\hat{R}\Psi)$ has a convergent contribution to each power of $z$.  This verification can be quite complicated; see Chapter 9, Section 3 of \cite{LeePandh} for an example.

\section{Feynman diagrams and integral formulas revisited}

In this section, we will attempt to generalize the integral formulas and their resulting Feynman diagram expansions computed in Chapter \ref{finite} to the infinite-dimensional case.  This is only interesting in the lower-triangular case, as the Feynman amplitude of any graph with at least one edge vanishes for upper-triangular operators.

To start, we must compute the analogues of the matrices $A$, $C$, and $D$ that describe the transformation in Darboux coordinates.  Suppose that
\[
R = \sum_{m \geq 0} R_m z^m
\]
is a lower-triangular symplectomorphism.  If
\[
e^{k, \alpha} = \phi^{\alpha}(-z)^{-1-k}, \;\;\;\; e^{\ell}_{\alpha} = \phi_{\alpha}z^{\ell},
\]
then it is easily check that
\[
\widetilde{e}^{k, \alpha} = R\cdot e^{k,\alpha} = \sum_{k' \geq 0}(-1)^{k-k'} (R_{k-k'})^{\alpha}_{\gamma}e^{k', \gamma} + \sum_{\ell' \geq 0}(-1)^{-1-k} (R_{\ell'+k+1})^{\alpha\gamma}e^{\ell'}_{\gamma},
\]
\[
\widetilde{e}^{\ell}_{\beta} = R(e^{\ell}_{\beta}) = \sum_{\ell' \geq 0} (R^*_{\ell'-\ell})_{\beta}^{\gamma} e^{\ell'}_{\gamma}.
\]

Let $\widetilde{p}_{k,\alpha}$ and $\widetilde{q}^{\beta}_{\ell}$ be defined by
\[
\sum_{k \geq 0} p_{k,\alpha} e^{k,\alpha} + \sum_{\ell \geq 0} q_{\ell}^{\beta} e^{\ell}_{\beta} = \sum_{k \geq 0} \widetilde{p}_{k,\alpha} \widetilde{e}^{k,\alpha} + \sum_{\ell \geq 0} \widetilde{q}_{\ell}^{\beta} \widetilde{e}^{\ell}_{\beta}.
\]
Then the relations among these coordinates are:
\begin{align*}
p_{k,\alpha} &= \sum_{k' \geq 0} (-1)^{k'-k} (R_{k'-k})^{\gamma}_{\alpha} \widetilde{p}_{k',\gamma},\\
q_{\ell}^{\beta} &= \sum_{k' \geq 0} (-1)^{-1-k'}(R_{\ell + k' + 1})^{\gamma\beta} \widetilde{p}_{k',\gamma} + \sum_{\ell' \geq 0} (R^*_{\ell - \ell'})_{\gamma}^{\beta} \widetilde{q}^{\gamma}_{\ell'}.
\end{align*}
That is, if we define matrices $\mathscr{A}$, $\mathscr{C}$, and $\mathscr{D}$ by
\[
(\mathscr{A}^*)_{(k,\alpha)}^{(k',\gamma)} = (-1)^{k' - k} (R_{k'-k})^{\gamma}_{ \alpha},
\]
\[
\mathscr{C}_{(\ell,\beta), (k', \gamma)} = (-1)^{-1-k'} (R_{\ell + k' + 1})^{\gamma\beta},
\]
\[
\mathscr{D}^{(\ell, \beta)}_{(\ell',\gamma)} = (R^*_{\ell - \ell'})_{\gamma}^{\beta},
\]
then the coordinates are related by
\begin{align*}
\mathbf{p} &= \mathscr{A} \widetilde{\mathbf{p}}\\
\mathbf{q} &= \mathscr{C}\widetilde{\mathbf{p}} + \mathscr{D}\widetilde{\mathbf{q}}.
\end{align*}
These matrices have rows and columns indexed by \[(k, \alpha) \in \Z^{\geq 0} \times \{1, \ldots, d\},\] but the entries vanish when $k$ is sufficiently large.

As in Section \ref{integrals} \ref{finite}, the integral formula will be expressed in terms of a function $\phi: \mathbb{H} \rightarrow \R$ defined by
\[\phi(q,p') = (\mathscr{D}^{-1}q) \cdot p' - \frac{1}{2}(\mathscr{D}^{-1}\mathscr{C}p') \cdot p'.\]
It is not necessary to invert $\mathscr{D}$, as one has $\mathscr{D}^{-1} = \mathscr{A}^T$ in this case.  Thus, the above gives an explicit formula for $\phi$.

Equipped with this, we would like to define
\begin{equation}
\label{infiniteintegral}
(U_R\psi)(q) = \lambda \int e^{\frac{1}{\hbar}(\phi(q,p') - q' \cdot p')} \psi(q') dq'  dp',
\end{equation}
where $\lambda$ is an appropriate normalization constant.  The problem with this, though, is that the domain of the variables $q = q_{\ell}^{\beta}$ and $p' = p'_{k,\alpha}$ over which we integrate is an infinite-dimensional vector space.  We have not specified a measure on this space, so it is not clear that (\ref{infiniteintegral}) makes sense.

Our strategy for making sense of (\ref{infiniteintegral}) will be to {\it define} it by its Feynman diagram expansion, as given by (\ref{Feynmanfinal}).  Modulo factors of $2\pi$, which are irrelevant because $U_R$ is defined only up to a real multiplicative constant, the answer is:
\[(U_R\psi)(q) = \frac{1}{\sqrt{\det(\mathscr{D}^{-1}\mathscr{C})}} \sum_{\mathbf{n} = (n_0, n_1, \ldots, )} \sum_{\Gamma \in G'(\mathbf{n})} \frac{\hbar^{-\chi_{\Gamma}}}{|\Aut(\Gamma)|} F_{\Gamma}(q),\]
where, as before, $G'(\mathbf{n})$ is the set of isomorphism classes of genus-labeled Feynman diagrams, and $F_{\Gamma}(q)$ is the genus-modified Feynman amplitude given by placing the $m$-tensor
\[\left.\sum_{\substack{\mathbf{m}\\ \sum m_{\ell,\beta} = m}} \frac{m!}{\prod m_{\ell,\beta}!} \frac{\d^m \mathcal{F}_g(s)}{\prod (\d s^{\ell}_{\beta})^{m_{\ell,\beta}}}\right|_{s = \mathscr{D}^{-1}q} \cdot \prod (s^{\beta}_{\ell})^{m_{\ell,\beta}}\]
at each $m$-valent vertex of genus $g$ and taking the contraction of tensors using the bilinear form $-\mathscr{D}^{-1}\mathscr{C}$.  Here, as before, $\mathcal{F}_g$ is defined by the expansion
\[\psi(q) = e^{\sum_{g \geq 0}\hbar^{g-1} \mathcal{F}_g(q)}.\]
Note that $\det(\mathscr{D}^{-1} \mathscr{C})$ is well-defined because, although these matrices have indices ranging over an infinite indexing set, they are zero outside of a finite range.

\section{Semi-classical limit}
\label{semiclassical}

The most important feature of quantization for our purposes that it relates higher genus information to genus-zero information.  We will return to this principle in the next chapter in the specific context of Gromov-Witten theory.  Before we do so, however, let us give a precise statement of this idea in the abstract setting of symplectic vector spaces.

Let $\mathbb{H}$ be a symplectic vector space, finite- or infinite-dimensional.  Fix a polarization $\mathbb{H} = \mathbb{H}_+ \oplus \mathbb{H}_-$, and suppose that for each $g \geq 0$,
\[\mathcal{F}_X^g: \mathbb{H}_+ \rightarrow \R, \;\; \mathcal{F}_Y^g: \mathbb{H}_+ \rightarrow \R\]
are functions on $\mathbb{H}_+$.  Package each of these two collections into the total descendent potentials:
\[\mathcal{D}_X = \exp\left(\sum_{g \geq 0} \hbar^{g-1} \mathcal{F}_X^g\right), \;\;\; \mathcal{D}_Y = \exp\left(\sum_{g \geq 0} \hbar^{g-1} \mathcal{F}_Y^g \right).\]
Let $\mathcal{L}_X$ and $\mathcal{L}_Y$ be the Lagrangian subspaces of $\mathbb{H}$ that, under the identification of $\mathbb{H}$ with $T^*\mathbb{H}_+$, coincide with the graphs of $d\mathcal{F}_X^0$ and $d\mathcal{F}_Y^0$, respectively.  That is,
\[\mathcal{L}_X = \{(p,q) \;| \; p = d_q \mathcal{F}_X^0\}\subset \mathbb{H},\]
and similarly for $\mathcal{L}_Y$.

\begin{thm}
Let $T$ be a symplectic transformation such that
\[
U_T \mathcal{D}_X = \mathcal{D}_Y.
\]
Then
\[T(\mathcal{L}_X) = \mathcal{L}_Y.\]
(The passage from $D^X$ to $\mathcal{L}_X$ is sometimes referred to as a {\bf semi-classical limit}.)

\begin{proof}
We will prove this statement in the finite-dimensional setting, but all of our arguments should carry over with only notational modifications to the infinite-dimensional case.  To further simplify, it suffices to prove the claim when $T$ is of one of the three basic types considered in Section 1.2.

{\bf Case 1}: $T = \left(\begin{array}{cc} I & B\\ 0 & I \end{array}\right)$.  Using the explicit formula for $U_T$ obtained in Chapter 1, the assumption $U_T \mathcal{D}_X = \mathcal{D}_Y$ in this case can be written as
\[
\exp\left(\sum_{g \geq 0} \hbar^{g-1}\mathcal{F}_Y^g\right) = \exp\left(\frac{1}{2\hbar} B_{\alpha\beta} q^{\alpha}q^{\beta}\right)\exp\left(\sum_{g \geq 0} \hbar^{g-1} \mathcal{F}^g_X\right).
\]
Taking logarithms of both sides, picking out the coefficient of $\hbar^{-1}$, and taking derivatives with respect to $q$, we find
\begin{equation}
\label{known}
d_q \mathcal{F}_Y^0 = Bq + d_q \mathcal{F}_X^0.
\end{equation}

Now, choose a point $(\overline{p},\overline{q}) \in \mathcal{L}_X$, so that $d_{\overline{q}} \mathcal{F}_X^0 = \overline{p}$.  Explicitly, the point in question is $\overline{p}_{\alpha}e^{\alpha} + \overline{q}^{\alpha}e_{\alpha}$, so its image under $T$ is
\[\overline{p}_{\alpha}\widetilde{e^{\alpha}} + \overline{q}^{\alpha} \widetilde{e_{\alpha}} = (\overline{p} + B\overline{q})_{\alpha}e^{\alpha} + \overline{q}^{\alpha} e_{\alpha},\]
using the expressions for $\widetilde{e^{\alpha}}$ and $\widetilde{e_{\alpha}}$ in terms of the $e$ basis obtained in Chapter 1.  Thus, the statement that $T(\overline{p}, \overline{q}) \in \mathcal{L}_Y$ is equivalent to
\[d_{\overline{q}}\mathcal{F}_Y^0 = \overline{p} + B\overline{q}.\]
Since $\overline{p} = d_{\overline{q}} \mathcal{F}_X^0$ by assumption, this is precisely equation (\ref{known}).  This proves that $T(\mathcal{L}_X) \subset \mathcal{L}_Y$, and the reverse inclusion follows from the analogous claim applied to $T^{-1}$.

{\bf Case 2}: $T = \left(\begin{array}{cc} A & 0 \\ 0 & A^{-T} \end{array}\right)$.  This case is very similar to the previous one, so we omit the proof.

{\bf Case 3}: $T = \left(\begin{array}{cc} I & 0\\ C & I \end{array}\right)$.  Consider the Feynman diagram expression for $U_T\mathcal{D}_X$ obtained in Section 1.5.  Up to a constant factor, the assumption $U_T \mathcal{D}_X = \mathcal{D}_Y$ in this case becomes
\begin{equation}
\label{Feyn}
\exp\left(\sum_{g \geq 0} \hbar^{g-1} \mathcal{F}_Y^g(q)\right) = \sum_{\Gamma} \frac{\hbar^{-\chi_{\Gamma}}}{|\text{Aut}(\Gamma)|} F_{\Gamma}^X(q),
\end{equation}
where $F_{\Gamma}^X(q)$ is the Feynman amplitude given by placing the $m$-tensor
\[\sum_{|\mathbf{m}|= m!} \frac{1}{m_1! \cdots m_n!}\left. \frac{\d^m \mathcal{F}_g^X}{(\d q_1')^{m_1} \cdots (\d q_n')^{m_n}}\right|_{q' = q} (q_1')^{m_1} \cdots (q_n')^{m_n}\]
at each $m$-valent vertex of genus $g$ in $\Gamma$ and taking the contraction of tensors using the bilinear form $-C$.

Recall that if one takes the logarithm of a sum over Feynman diagrams, the result is a sum over connected graphs only, and in this case, $-\chi_{\Gamma} = g-1$.  Thus, if one takes the logarithm of both sides of (\ref{Feyn}) and picks out the coefficient of $\hbar^{-1}$, the result is
\[
\mathcal{F}_Y^0(q) = \sum_{\Gamma \text{ connected, genus } 0} \frac{F_{\Gamma}(q)}{|\text{Aut}(\Gamma)|}.
\]

Let us give a more explicit formulation of the definition of $F_{\Gamma}(q)$.  For convenience, we adopt the notation of Gromov-Witten theory and write
\[\mathcal{F}_0^X(q) = \langle \; \rangle + \langle e_{\alpha} \rangle^X q^{\alpha} + \frac{1}{2}\langle e_{\alpha}, e_{\beta} \rangle^X q^{\alpha}q^{\beta} + \cdots = \langle \langle \; \rangle \rangle^X(q),\]
where $\langle \langle \phi_1, \ldots, \phi_n \rangle \rangle^X(q) := \sum_{k \geq 0} \frac{1}{k!} \langle \phi_1, \ldots, \phi_n, q, \ldots, q \rangle^X$ ($k$ copies of $q$) and the brackets are defined by the above expansion.  Then derivatives of $\mathcal{F}^0_X$ are given by adding insertions to the double bracket.  It follows that
\[F_{\Gamma}^X(q) = \sum_{\{i_h\}} \prod_{v \in V(\Gamma)} \left\langle \left\langle \prod_{h \in H(v)} e_{i_h} \right\rangle\right\rangle^X(q) \prod_{e = (a,b) \in E(\Gamma)}(-C^{i_a i_b}).\]
Here, $V(\Gamma)$ and $E(\Gamma)$ denote the vertex sets and edge sets of $\Gamma$, respectively, while $H(v)$ denotes the set of half-edges associated to a vertex $v$.  The summation is over all ways to assign an index $i_h \in \{1, \ldots, d\}$ to each half-edge $h$, where $d$ is equal to the dimension of $\mathbb{H}_+$.  For an edge $e$, we write $e = (a,b)$ if $a$ and $b$ are the two half-edges comprising $e$.

Thus, we have re-expressed the relationship between $\mathcal{F}_0^Y$ and $\mathcal{F}_0^X$ as
\[\mathcal{F}_0^Y(q) = \sum_{\Gamma \text{ connected, genus } 0} \frac{1}{|\text{Aut}(\Gamma)|}  \sum_{\{i_h\}} \prod_{v \in V(\Gamma)} \left\langle \left\langle \prod_{h \in H(v)} e_{i_h} \right\rangle\right\rangle^X(q) \prod_{e = (a,b) \in E(\Gamma)}(-C^{i_a i_b}).\]

Now, to prove the claim, choose a point $(\overline{p}, \overline{q}) \in \mathcal{L}_Y$.  We will prove that $T^{-1}(\overline{p}, \overline{q}) \in \mathcal{L}_X$.  Applying the same reasoning used in Case 1 to the inverse matrix $T^{-1} = \left(\begin{array}{cc} I & 0 \\ -C & I \end{array}\right)$, we find that
\[T^{-1}(\overline{p}, \overline{q}) = \overline{p}_{\alpha}e^{\alpha} + (-C \overline{p} + \overline{q})^{\alpha} e_{\alpha}.\]
Therefore, the claim is equivalent to
\[d_{\overline{q}} \mathcal{F}_0^Y = d_{-Cd_{\overline{q}} \mathcal{F}_0^Y + \overline{q}} \mathcal{F}_0^X.\]

From the above, one finds that the $i$th component of the vector $d_{\overline{q}}\mathcal{F}_Y^0$ is equal to
\[ \sum_{\Gamma} \frac{1}{|\text{Aut}(\Gamma)|}  \sum_{\{i_h\}, w \in V(\Gamma)} \left\langle\left\langle e_i \prod_{h \in H(w)} e_{i_h} \right\rangle\right\rangle^X(\overline{q}) \prod_{v \neq w \in V(\Gamma)} \left\langle \left\langle \prod_{h \in H(v)} e_{i_h} \right\rangle\right\rangle^X(\overline{q}) \prod_{e = (a,b) \in E(\Gamma)}(-C^{i_a i_b}).\]
On the other hand, the same equation shows that the $i$th component of the vector $d_{-Cd_{\overline{q}}\mathcal{F}_0^Y + \overline{q}} \mathcal{F}_X^0$ is equal to
\begin{align*}
&\langle \langle e_i \rangle \rangle^X(-Cd_{\overline{q}}\mathcal{F}_0^Y + \overline{q})\\
=&\langle\langle e_i \rangle \rangle^X\left(\sum_{\Gamma} \frac{1}{|\text{Aut}(\Gamma)|}  \sum_{\{i_h\}, w \in V(\Gamma),j,k} \left\langle\left\langle e_j \prod_{h \in H(w)} e_{i_h} \right\rangle\right\rangle^X(\overline{q}) \prod_{v \neq w \in V(\Gamma)} \left\langle \left\langle \prod_{h \in H(v)} e_{i_h} \right\rangle\right\rangle^X(\overline{q})\right.\\
&\hspace{5cm}\left.\prod_{e = (a,b) \in E(\Gamma)}(-C^{i_a i_b})(-C)^{kj}e_k + q^{\ell}e_{\ell}\right)\\
&=\sum_{\Gamma_1, \ldots, \Gamma_n} \frac{1}{n! |\text{Aut}(\Gamma_1)| \cdots |\text{Aut}(\Gamma_n)|} \sum_{\substack{\{i_h\}\\w_1 \in V(\Gamma_1), \ldots, w_n \in V(\Gamma_n)\\j_1, \ldots, j_n\\k_1, \ldots, k_n\\\ell_1, \ldots, \ell_m}} \prod_{c=1}^n \left\langle \left\langle  e_{j_c}\prod_{h \in H(w_c)} e_{i_h}\right\rangle\right\rangle^X(\overline{q})\\
&\hspace{0.5cm} \prod_{v \neq w_1, \ldots, w_n} \left\langle \left\langle \prod_{h \in H(v)} e_{i_h} \right\rangle\right\rangle^X(\overline{q})\prod_{e=(a,b)} (-C^{i_ai_b}) (-C)^{k_c j_c} \frac{\langle e_i, e_{k_1}, \ldots, e_{k_n}, q^{\ell_1}e_{\ell_1}, \ldots, q^{\ell_m}e_{\ell_m}\rangle}{m!}
\end{align*}

\begin{align*}
&=\sum_{\Gamma_1, \ldots, \Gamma_n} \frac{1}{n! |\text{Aut}(\Gamma_1)| \cdots |\text{Aut}(\Gamma_n)|} \sum_{\substack{\{i_h\}\\w_1 \in V(\Gamma_1), \ldots, w_n \in V(\Gamma_n)\\j_1, \ldots, j_n\\k_1, \ldots, k_n}} \prod_{c=1}^n \left\langle \left\langle  e_{j_c}\prod_{h \in H(w_c)} e_{i_h}\right\rangle\right\rangle^X(\overline{q})\\
&\hspace{2cm} \prod_{v \neq w_1, \ldots, w_n} \left\langle \left\langle \prod_{h \in H(v)} e_{i_h} \right\rangle\right\rangle^X(\overline{q})\prod_{e=(a,b)} (-C^{i_ai_b}) (-C)^{k_c j_c}\langle \langle e_i, e_{k_1}, \ldots, e_{k_n}\rangle\rangle(\overline{q})
\end{align*}
Upon inspection, this is equal to the sum of all ways of starting with a distinguished vertex (where $e_i$ is located) and adding $n$ spokes labeled $k_1, \ldots, k_n$, then attaching $n$ graphs $\Gamma_1, \ldots, \Gamma_n$ to this vertex via half-edges labeled $j_1, \ldots, j_n$.  This procedures yields all possible graphs with a distinguished vertex labeled by $e_i$---the same summation that appears in the expression for $d_{\overline{q}}\mathcal{F}_0^Y$.  Each total graph appears in multiple ways, corresponding to different ways of partitioning it into subgraphs labeled $\Gamma_1, \ldots, \Gamma_n$.  However, it is a combinatorial exercise to verify that, with this over-counting, the automorphism factor in front of each graph $\Gamma$ is precisely $\frac{1}{|\text{Aut}(\Gamma)|}$.

Thus, we find that $d_{-Cd_{\overline{q}}\mathcal{F}_0Y + \overline{q}}\mathcal{F}_0^X = d_{\overline{q}} \mathcal{F}_0^Y$, as required.
\end{proof}
\end{thm}

\chapter{Applications of quantization to Gromov-Witten theory} \label{GW}

In this final chapter, we will return to the situation in which $\mathbb{H} = H^*(X; \Lambda)((z^{-1}))$ for $X$ a projective variety.  In Section \ref{equations}, we show that many of the basic equations of Gromov-Witten theory can be expressed quite succinctly as equations satisfied by the action of a quantized operator on the total descendent potential.  More strikingly, according to Givental's conjecture, there is a converse in certain special cases to the semi-classical limit statement explained in Section \ref{semiclassical}; we discuss this in Section \ref{Givental} below.  In Section \ref{twistedtheory}, we briefly outline the machinery of twisted Gromov-Witten theory developed by Coates and Givental.  This is a key example of the way quantization can package complicated combinatorics into a manageable formula.

Ultimately, these notes only scratch the surface of the applicability of the quantization machinery to Gromov-Witten theory.  There are many other interesting directions in this vein, so we conclude the book with a brief overview of some of the other places in which quantization arises.  The interested reader can find much more in the literature.

\section{Basic equations via quantization}
\label{equations}

Here with give a simple application of quantization as a way to rephrase some of the axioms of Gromov-Witten theory.  This section closely follows Examples 1.3.3.2 and 1.3.3.3 of \cite{Coates}.

\subsection{String equation}

Recall from (\ref{trunc}) that the string equation can be expressed as follows:
\begin{align*}
\sum_{g,n,d} \frac{Q^d\hbar^{g-1}}{(n-1)!}\langle 1, \mathbf{t}(\psi), \ldots, \mathbf{t}(\psi)\rangle_{g,n,d}^X &= \sum_{g,n,d}  \frac{Q^d\hbar^{g-1}}{(n-1)!}\left\langle \left[ \frac{\mathbf{t}(\psi)}{\psi}\right]_+ \hspace{-0.25cm},\mathbf{t}(\psi), \ldots, \mathbf{t}(\psi) \right\rangle_{g,n,d}^X\\
&+\frac{1}{2\hbar}\langle t_0, t_0 \rangle.
\end{align*}
Applying the dilaton shift (\ref{dilatonshift}), this is equivalent to
\[-\frac{1}{2\hbar}(q_0, q_0) - \sum_{g,n,d} \frac{Q^d\hbar^{g-1}}{(n-1)!} \left\langle \left[ \frac{\mathbf{q}(\psi)}{\psi}\right]_+, \mathbf{t}(\psi), \ldots, \mathbf{t}(\psi) \right\rangle^X_{g,n,d}=0.\]
From Example \ref{Bzm} with $m=-1$, one finds
\begin{align*}
\widehat{\frac{1}{z}} &= -\frac{1}{2\hbar}\langle q_0, q_0\rangle - \sum_k q^{\alpha}_{k+1} \frac{\d}{\d q^{\alpha}_k}\\
&=-\frac{1}{2\hbar} \langle q_0, q_0 \rangle - \d_{1/z}.
\end{align*}
Thus, applying (\ref{qplus}), it follows that the string equation is equivalent to
\[\widehat{\frac{1}{z}} \mathcal{D}_X = 0.\]

%


\subsection{Divisor equation}

In a similar fashion, the divisor equation can be expressed in terms of a quantized operator.  Summing over $g,n,d$ and separating the exceptional terms, equation (\ref{divisoreqn}) can be stated as
\begin{align*}
\sum_{g,n,d}\frac{Q^d\hbar^{g-1}}{(n-1)!} \langle \mathbf{t}(\psi), \ldots, \mathbf{t}(\psi), \rho\rangle_{g,n,d}^X
&=\sum_{g,n,d} \frac{Q^d\hbar^{g-1}}{n!} (\rho,d) \langle \mathbf{t}(\psi), \ldots, \mathbf{t}(\psi)\rangle_{g,n,d}\\
\hspace{4.5cm}&+\sum_{g,n,d} \frac{Q^d\hbar^{g-1}}{(n-1)!}\left\langle \left[ \frac{\rho \mathbf{t}(\psi)}{\psi}\right]_+, \mathbf{t}(\psi), \ldots, \mathbf{t}(\psi) \right\rangle_{g,n,d}^X
\end{align*}
\begin{equation}
\label{div}
\hspace{5cm}+\frac{1}{2\hbar}\langle \mathbf{t}(\psi), \mathbf{t}(\psi), \rho\rangle_{0,3,0} + \langle \rho \rangle_{1,1,0}.
\end{equation}

The left-hand side and the second summation on the right-hand side combine to give
\[
- \sum_{g,n,d} \frac{Q^d\hbar^{g-1}}{(n-1)!} \left\langle \left[\frac{\rho(\mathbf{t}(\psi) - \psi)}{\psi}\right]_+, \mathbf{t}(\psi), \ldots, \mathbf{t}(\psi) \right\rangle_{g,n,d}^X.
\]
After the dilaton shift, this is equal to $\d_{\rho/z}(\sum \hbar^{g-1} \mathcal{F}_g^X)$.

As for the first summation, let $\tau_1, \ldots, \tau_r$ be a choice of basis for $H_2(X;\Z)$, which yields a set of generators $Q_1, \ldots, Q_r$ for the Novikov ring.  Write
\[\rho = \sum_{i=1}^r \rho_i \tau^i\]
in the dual basis $\{\tau^i\}$ for $\{\tau_i\}$.  Then the first summation on the right-hand side of (\ref{div}) is equal to
\[
\sum_{i=1}^r \rho_i Q_i \frac{\d}{\d Q_i} \left(\sum_g\hbar^{g-1} \mathcal{F}_g^X\right).
\]

The first exceptional term is computed as before:
\[\langle \mathbf{t}(\psi), \mathbf{t}(\psi), \rho\rangle_{0,3,0} = (q_0\rho, q_0).\]
The second exceptional term is more complicated in this case.  We require the fact that
\[\M_{1,1}(X, 0) \cong X \times \M_{1,1},\]
and that under this identification
\[[\M_{1,1}(X,d)]^{\text{vir}} = e(T_X \times \mathcal{L}_1^{-1}) = e(T_X) - \psi_1 c_{D-1}(T_X),\]
where $\mathcal{L}_1$ is the cotangent line bundle (whose first Chern class is $\psi_1$) and $D = \text{dim}(X)$.  Thus,
\[
\langle \rho \rangle_{1,1,0}^X = \int_X e(T_X)\cdot \rho - \int_{\M_{1,1}} \psi_1 \int_X c_{D-1}(T_X) \cdot \rho = -\frac{1}{24} \int_X c_{D-1}(T_X) \cdot \rho.
\]

Putting all of these pieces together, we can express the divisor equation as
\begin{align*}
-\frac{1}{2\hbar}(q_0\rho, q_0) - \d_{\rho/z} &\left(\sum_g\hbar^{g-1} \mathcal{F}_g^X\right)\\
&= \left(\sum_i \rho_i Q_i \frac{\d}{\d Q_i} - \frac{1}{24}\int_X c_{D-1}(T_X) \cdot \rho\right)\left(\sum_g\hbar^{g-1} \mathcal{F}_g^X\right),
\end{align*}
or in other words, as
\[
\widehat{\left(\frac{\rho}{z}\right)}\cdot\D_X = \left(\sum_i \rho_i Q_i \frac{\d}{\d Q_i} - \frac{1}{24}\int_X c_{D-1}(T_X) \cdot \rho\right)\D_X.
\]
It should be noted that the left-hand side of this equality makes sense because multiplication by $\rho$ is a self-adjoint linear transformation on $H^*(X)$ under the Poincar\'e pairing, and hence multiplication by $\rho/z$ is an infinitesimal symplectic transformation.

\section{Givental's conjecture}
\label{Givental}

The material in this section can be found in \cite{GiventalSemiSimple} and \cite{Lee}.

Recall from Section \ref{axiomatic} that an axiomatic genus zero theory is a symplectic vector space $\mathbb{H} = H((z^{-1}))$ together with a formal function $G_0(\mathbf{t})$ satisfying the differential equations corresponding to the string equation, dilaton equation, and topological recursion relations in genus zero.

The {\bf symplectic} (or {\bf twisted}) {\bf loop group} is defined as the set $\{M(z)\}$ of $\text{End}(H)$-valued formal Laurent series in $z^{-1}$ satisfying the symplectic condition $M^*(-z)M(z) = I$.  There is an action of this group on the collection of axiomatic genus zero theories.  To describe the action, it is helpful first to reformulate the definition of an axiomatic theory in a more geometric, though perhaps less transparent, way. 

Associated to an axiomatic genus zero theory is a Lagrangian subspace
\[\mathcal{L} = \{(\mathbf{p}, \mathbf{q}) \; | \; \mathbf{p} = d_{\mathbf{q}} G_0 \} \subset \mathbb{H},\]
where $(\mathbf{p}, \mathbf{q})$ are the Darboux coordinates on $\mathbb{H}$ defined by (\ref{darboux}) and we are identifying $\mathbb{H} \cong T^*\mathbb{H}_+$ by way of this polarization.  Note, here, that $G_0(\mathbf{t})$ is identified with a function of $\mathbf{q} \in \mathbb{H}_+$ via the dilaton shift, and that this is the same Lagrangian subspace discussed in Section \ref{semiclassical}.

According to Theorem 1 of \cite{GiventalFrobenius}, a function $G_0(\mathbf{t})$ satisfies the requisite differential equations if and only if the corresponding Lagrangian subspace $\mathcal{L}$ is a Lagrangian {\it cone} with the vertex at the origin satisfying
\[\mathcal{L} \cap T_{\mathbf{f}} \mathcal{L} = z T_{\mathbf{f}} \mathcal{L}\]
for each $\mathbf{f} \in \mathcal{L}$.

The symplectic loop group can be shown to preserve these properties.  Thus, an element $T$ of the symplectic loop group acts on the collection of axiomatic theories by sending the theory with Lagrangian cone $\mathcal{L}$ to the theory with Lagrangian cone $T(\mathcal{L})$.

There is one other equivalent formulation of the definition of axiomatic genus-zero theories, in terms of abstract Frobenius manifolds.  Roughly speaking, a Frobenius manifold is a manifold equipped with a product on each tangent space that gives the tangent spaces the algebraic structure of Frobenius algebras.  A Frobenius manifold is called {\bf semisimple} if, on a dense open subset of the manifold, these algebras are semisimple.  This yields a notion of semisimplicity for axiomatic genus zero theories.  Given this, we can formulate the statement of the symplectic loop group action more precisely:

\begin{thm} \cite{GiventalFrobenius}
The symplectic loop group acts on the collection of axiomatic genus-zero theories.  Furthermore, the action is transitive on the semisimple theories of a fixed rank $N$.
\end{thm}

Here, the {\bf rank} of a theory is the rank of $H$.

The genus-zero Gromov-Witten theory of a collection of $N$ points gives a semisimple axiomatic theory of rank $N$, which we denote by $H_{N}$.  The theorem implies that any semisimple axiomatic genus-zero theory $T = (\mathbf{H}, G_0)$ can be obtained from $H_N$ by the action of an element of the twisted loop group.   Via the process of Birkhoff factorization, such a transformation can be expressed as $S \circ R$ in which $S$ is upper-triangular and $R$ is lower-triangular.

\begin{df}
The {\bf axiomatic $\tau$-function} of an axiomatic theory $T$ is defined by
\[\tau^T_G = \hat{S} (\hat{R} \mathcal{D}_{N}),\]
where $S \circ R$ is the element of the symplectic loop group taking the theory $H_N$ of $N$ points to $T$, and $\mathcal{D}_N$ is the total descendent potential for the Gromov-Witten theory of $N$ points.
\end{df}

If $T$ is in fact the genus-zero Gromov-Witten theory of a space $X$, then we have two competing definitions of the higher-genus potential: $\mathcal{D}_X$ and $\tau^T_G$.  Givental's conjecture is the statement that these two agree:

\begin{conj}[Givental's conjecture \cite{GiventalSemiSimple}] \label{giventalconj}  If $T$ is the semisimple axiomatic theory corresponding to the genus-zero Gromov-Witten theory of a projective variety $X$, then $\tau^T_G = \mathcal{D}_X$.
\end{conj}

In other words, the conjecture posits that in the semisimple case, if an element of the symplectic loop group matches two genus zero theories, then its quantization matches their total descendent potentials.  Because the action of the symplectic loop group is transitive on semisimple theories, this amounts to a classificaiton of all higher-genus theories for which the genus-zero theory is semisimple.

Givental proved his conjecture in case $X$ admits a torus action and the total descendent potentials are taken to be the equivariant Gromov-Witten potentials.  In 2005 Teleman announced a proof of the conjecture in general:

\begin{thm}\cite{Teleman}
Givental's conjecture holds for any semisimple axiomatic theory.
\end{thm}

One important application of Givental's conjecture is the proof of the Virasoro conjecture in the semisimple case.  The conjecture states:

\begin{conj}
For any projective manifold $X$, there exist ``Virasoro operators" $\{\widehat{L^X_m}\}_{m\geq -1}$ satisfying the relations
\begin{equation}
\label{virasororel}
[\widehat{L^X_m}, \widehat{L^X_n}] = (m-n)\widehat{L^X_{m+n}},
\end{equation}
such that
\[\widehat{L^X_m}\mathcal{D}_X = 0\]
for all $m \geq -1$.
\end{conj}

In the case where $X$ is a collection of $N$ points, the conjecture holds by setting $\widehat{L^X_m}$ equal to the quantization of
\[L_m := -z^{-1/2}D^{m+1} z^{-1/2},\]
where
\[D:= z\left(\frac{d}{dz}\right)z = z^2\frac{d}{dz} + z.\]
The resulting operators $\{\widehat{L_m}\}_{m \geq -1}$ are the same as $N$ copies of those used in Witten's conjecture  \cite{Witten}, and the relations (\ref{virasororel}) indeed hold for these operators.

Thus, by Witten's conjecture, the Virasoro conjecture holds for any semisimple Gromov-Witten theory by setting
\[\widehat{L^X_m} = \widehat{S} (\widehat{R} \widehat{L_m} \widehat{R}^{-1}) \widehat{S}^{-1} \]
for the transformation $S \circ R$ taking the theory of $N$ points to the Gromov-Witten theory of $X$.\footnote{In fact one must check that $\widehat{L^X_m}$ defined this way agrees with the Virasoro operators of the conjecture, but this can be done.}

\section{Twisted theory}
\label{twistedtheory}

The following is due to Coates and Givental; we refer the reader to the exposition presented in \cite{Coates}.

Let $X$ be a projective variety equipped with a holomorphic vector bundle $E$.  Then $E$ induces a $K$-class on $\M_{g,n}(X,d)$,
\[E_{g,n,d} = \pi_*f^*E \in K^0(\M_{g,n}(X,d)),\]
where
\[\xymatrix{
\mathscr{C} \ar[r]^{f} \ar[d]_{\pi} & X\\
\M_{g,n}(X,d), & 
}\]
is the univeral family over the moduli space.  Consider an invertible multiplicative characteristic class
\[
c: K^0(\M_{g,n}(X,d)) \rightarrow H^*(\M_{g,n}(X,d)).
\]
Any such class can be written in terms of Chern characters
\[c(\cdot) = \exp\left(\sum_{k \geq 0} s_k \text{ch}_k(\cdot)\right),\]
for some parameters $s_k$.

A {\bf twisted Gromov-Witten invariant} is defined as
\[\langle \tau_1(\gamma_1) \cdots \tau_n(\gamma_n) ; c(E_{g,n,d}) \rangle^X_{g,n,d} = \int_{[\M_{g,n}(X, d)]^{\text{vir}}} \hspace{-1.5cm}ev_1^*(\gamma_1) \psi_1^{a_1} \cdots ev_n^*(\gamma_n) \psi_n^{a_n} c(E_{g,n,d}).\]
These fit into a twisted genus-$g$ potential $\mathcal{F}^g_{c, E}$ and a twisted total descendent potential $\mathcal{D}_{c,E}$ in just the way that the usual Gromov-Witten invariants do.

There is also a Lagrangian cone $\mathcal{L}_{c,E}$ associated to the twisted theory, but a bit of work is necessary in order to define it.  The reason for this is that the Poincar\'e pairing on $H(X; \Lambda)$ should be given by three-point correlators.  As a result, when we replace Gromov-Witten invariants by their twisted versions we must modify the Poincar\'e pairing, and hence the symplectic structure on $\mathbb{H}$, accordingly.  Denote this modified symplectic vector space by $\mathbb{H}_{c,E}$.  There is a symplectic isomorphism
\[\mathbb{H}_{c,E} \rightarrow \mathbb{H}\]
\[x \mapsto \sqrt{c(E)} x.\]
We define the Lagrangian cone $\mathcal{L}_{c,E}$ of the twisted theory by
\[\mathcal{L}_{c,E} = \sqrt{c(E)} \cdot \{ (\mathbf{p}, \mathbf{q}) \; | \; \mathbf{p} = d_{\mathbf{q}} \mathcal{F}^0_{c,E}\} \subset \mathbb{H},\]
where we use the usual dilaton shift to identify $\mathcal{F}^0_{c,E}(\mathbf{t})$ with a function of $\mathbf{q} \in (\mathbb{H}_{c,E})_+$.

The quantum Riemann-Roch theorem of Coates-Givental gives an expression for $\mathcal{D}_{c,E}$ in terms of a quantized operator acting on the untwisted Gromov-Witten descendent potential $\mathcal{D}_X$ of $X$: 

\begin{thm}\cite{Coates}
\label{twisted}
The twisted descendent potential is related to the untwisted descendent potential by
\begin{align*}
&\exp\left(-\frac{1}{24}\sum_{\ell > 0} s_{\ell -1} \int_X \text{ch}_{\ell}(E) c_{D-1}(T_X)\right)\exp\left(\frac{1}{2}\int_X e(X) \wedge \left(\sum_{j \geq 0} s_j \text{ch}_j(E)\right)\right) \mathcal{D}_{c,E}\\
=&\exp\left(\sum_{\substack{m > 0\\ \ell \geq 0}} s_{2m-1 + \ell} \frac{B_{2m}}{(2m)!} (\text{ch}_{\ell}(E)z^{2m-1})^{\wedge}\right) \exp\left(\sum_{\ell > 0} s_{\ell-1} (\text{ch}_{\ell}(E)/z)^{\wedge}\right) \mathcal{D}_X.
\end{align*}
Here $B_{2m}$ are the Bernoulli numbers.
\end{thm}

The basic idea of this theorem is to write
\begin{align*}
c(E_{g,n,d}) &= \exp\left(\sum_{k \geq 0} s_k \text{ch}_k(R\pi_*f^*E)\right)\\
 &= \exp\left(\sum_{k \geq 0} s_k \big(\pi_*(\text{ch}(f^*E) \text{Td}^{\vee}(T_{\pi}))\big)_k \right),
\end{align*}
using the Grothendieck-Riemann-Roch formula.  A geometric theorem expresses $\text{Td}^{\vee}(T_{\pi})$ in terms of $\psi$ classes on various strata of the moduli space, and the rest of the proof of Theorem \ref{twisted} is a difficult combinatorial exercise in keeping track of these contributions.

Taking a semi-classical limit and applying the result discussed in Section \ref{semiclassical} of the previous chapter, we obtain:

\begin{crl}
The Lagrangian cone $\mathcal{L}_{c,E}$ satisfies
\[\mathcal{L}_{c,E} = \exp\left(\sum_{m \geq 0} \sum_{0 \leq \ell \leq D} s_{2m-1+\ell} \frac{B_{2m}}{(2m)!} \text{ch}_{\ell}(E) z^{2m-1}\right)\mathcal{L}_X.\]
\end{crl}

This theorem and its corollary are extremely useful even when one is only concerned with the genus zero statement.  For example, it is used in the proof of the Landau-Ginzburg/Calabi-Yau correspondence in \cite{ChR}.  In that context, the invariants under consideration, known as FJRW invariants, are given by twisted Gromov-Witten invariants {\it only} in genus zero.  Thus, Theorem \ref{twisted} actually says nothing about higher-genus FJRW invariants.  Nevertheless, an attempt to directly apply Grothendieck-Riemann-Roch in genus zero to obtain a relationship between FJRW invariants and untwisted invariants is combinatorially unmanageable; thus, the higher-genus statement, while not directly applicable, can be viewed as a clever device for keeping track of the combinatorics of the Grothendieck-Riemann-Roch computation.

\section{Concluding remarks}

There are a number of other places in which quantization proves useful for Gromov-Witten theory.  For example, it was shown in \cite{30}, \cite{31}, \cite{14} that relations in the so-called tautological ring, an important subring of $H^*(\M_{g,n}(X, \beta))$, are invariant under the action of the symplectic loop group.  This was used to give one proof of Givental's Conjecture in genus $g \leq 2$ (see \cite{10} and \cite{34}), and can also be used to derive tautological relations (see \cite{29}, \cite{1}, and \cite{2}).  We refer the interested reader to \cite{Lee} for a summary of these and other applications of quantization with more complete references.

\bibliographystyle{abbrv}
\nocite{*}
\bibliography{biblio}

\end{document}